\documentclass[leqno]{article}
\usepackage[dvips]{graphicx}
\usepackage{amsxtra}
\usepackage{amssymb}
\usepackage{amscd}
\usepackage{array}
\usepackage{theorem}

\usepackage{latexsym}
\theorembodyfont{\rmfamily}
\newtheorem{example}{Example}[section]

\newtheorem{remark}[example]{Remark}
\newtheorem{definition}[example]{Definition}
\theorembodyfont{\itshape}
\newtheorem{theorem}[example]{Theorem}
\newtheorem{proposition}[example]{Proposition}
\newtheorem{corollary}[example]{Corollary}
\newtheorem{lemma}[example]{Lemma}

\font\ddpp=msbm10  scaled \magstep 1  

\newenvironment{proof}{\noindent{\bf Proof~:}}{\QED\medskip}
\def\QED{\hskip0.1em\hfill\null\ \null\nobreak\hfill
\kern3pt\lower1.8pt\vbox{\hrule\hbox {\vrule\kern1pt\vbox{\kern1.7pt
\hbox{$\scriptstyle QED$}\kern0.2pt}\kern1pt\vrule}\hrule}}

\def\R{\hbox{\ddpp R}}         

\def\frg{{\frak g}}
\def\frh{{\frak h}}

\def\Re{{\frak R}{\frak e}\,}
\def\Im{{\frak I}{\frak m}\,}

\begin{document}

\title{\bf Contact 5-manifolds with $SU(2)$-structure}
\author{Luis C. de Andr\'es, Marisa Fern\'andez, Anna Fino and Luis Ugarte}

\date{\today}

\maketitle

\begin{abstract}
We consider $5$-manifolds with a contact form arising from a hypo
structure \cite{ConS}, which we call \emph{hypo-contact}. We provide
existence conditions for  such a structure on an oriented
hypersurface of a $6$-manifold with a half-flat $SU(3)$-structure.
For half-flat manifolds with a Killing vector field $X$ preserving
the $SU(3)$-structure we study the geometry of the orbits space.
Moreover, we describe the solvable Lie algebras admitting a
\emph{hypo-contact} structure. This allows us to exhibit examples of
Sasakian $\eta$-Einstein manifolds, as well as to prove that such
structures give rise to new metrics with holonomy $SU(3)$ and $G_2$.
\end{abstract}

\bigskip

\section{Introduction}

Recently, Conti and Salamon introduced in \cite{ConS} \emph{hypo
structures} on $5$-manifolds as a generalization in dimension $5$ of
Sasakian-Einstein metrics; indeed, Sasakian-Einstein metrics
correspond to Killing spinors and {\em hypo structures} are induced
by {\em generalized} Killing spinors. In terms of differential
forms, a {\em hypo structure} on a $5$-manifold $N$ is determined by
a quadruplet $(\eta,\omega_i, 1\leq i \leq 3)$ of differential
forms, where $\eta$ is a nowhere vanishing $1$-form and $\omega_i$
are $2$-forms on $N$ satisfying certain relations (see \eqref{rhypo}
in Section~\ref{hypo-halff}).

If the forms $\eta$ and $\omega_i$ satisfy
$$
  d\eta=-2\omega_3,  \quad d\omega_1=3\eta\wedge\omega_2,   \quad
  d\omega_2=-3\eta\wedge\omega_1,
$$
then $N$ is a Sasakian-Einstein manifold, that is, a Riemannian
manifold such that $N\times\mathbb R$ with the cone metric is
K\"ahler and Ricci flat \cite{BGal}. Thus $N\times\mathbb R$ has
holonomy contained in $SU(3)$ or, equivalently, it has an integrable
$SU(3)$-structure which means that there is an almost Hermitian
structure, with K\"ahler form $F$, and a $(3,0)$-form $\Psi =
\Psi_{+} + i \Psi_{-}$ on $N^5\times\mathbb R$ satisfying
$dF=d\Psi_+=d\Psi_-=0$. In the general case of a hypo structure, in
\cite{ConS} it is proved that a real analytic hypo structure on a
real analytic $5$-manifold $N$ can be lifted to an integrable
$SU(3)$-structure on $N\times I$, for some open interval $I$ or
equivalently that  $(\eta,\omega_i, 1\leq i \leq 3)$ belongs to a
one-parameter family of hypo structures $(\eta(t),\omega_i(t), 1\leq
i \leq 3)$ satisfying the evolution equations \eqref{vol} given in
Section \ref{hypo-halff}. Without assuming real analyticity no
general result is known.  Conversely, any oriented hypersurface of a
$6$-manifold with an integrable $SU(3)$-structure is naturally
endowed with a hypo structure (see Section \ref{hypo-halff} for
details).

In general, for a hypo $5$-manifold the $1$-form $\eta$ is not a
contact form. In this paper we deal with $5$-manifolds $N$ having a
{\em hypo-contact structure}, that is, a hypo structure
$(\eta,\omega_1,\omega_2,\omega_3)$ such that $d\eta=-2\omega_3$, so
$\eta$ is a contact form on $N$. Such structures were considered by
Conti   in \cite{Con} and  by Bedulli and Vezzoni in  \cite{BV},
where an explicit expression for the Ricci and  scalar curvature  is
given  in terms of torsion forms and its derivatives.

If we weaken the integrability condition of the $SU(3)$-structure
 $(F, \Psi_+, \Psi_-)$ on the $6$-manifold $M$ to be
half-flat in the sense of \cite{CS}, i.e.  $F\wedge F$ and $\Psi_+$
are closed, Hitchin in \cite{H} proved that there is a
$G_2$-structure on $M\times I$ with holonomy contained in $G_2$ if
the half-flat structure $(F,\Psi_+,\Psi_-)$ is such that certain
evolution equations admit a solution $(F(t),\Psi_+(t),\Psi_-(t))$,
for all real parameter $t$ lying in some   interval $I$,
 with $F(0)=F$, $\Psi_+(0)=\Psi_+$ and
$\Psi_-(0)=\Psi_-$.

Regarding {\em hypo-contact structures}, in Theorem
\ref{converse-hypoc-half-flat-0}  we provide conditions which imply
that there is a hypo-contact structure on any oriented hypersurface
$f\colon N\longrightarrow M$ of a half-flat manifold $M$; and when
$M$ has a Killing vector field preserving the $SU(3)$-structure, we
study the geometry of the orbits space. Moreover, in Proposition
\ref{from-hypo-to-half-flat} we  show  how to lift a hypo structure
on  a $5$-manifold $N$ to a half-flat structure on the total space
of a a  circle bundle over $N$.

Our main results concern
 solvable Lie groups of dimension $5$ with a left-invariant
hypo-contact structure. In particular, using such structures and
solving the corresponding evolution equations, we
 construct new metrics with holonomy $SU(3)$ and $G_2$.
In Section \ref{hypo-contact-solvable} the classification of
solvable Lie algebras with a hypo-contact structure is given,
showing the following theorem.

\begin{theorem}\label{clasification}
A $5$-dimensional solvable Lie algebra admits a hypo-contact
structure if and only if it is isomorphic to one of the following:
$$
\begin{array}{rcl}
\mathfrak h_{1} &\!\!\!\colon\!\!\!& [X_1,X_4]=[X_2,X_3]=X_5 ;\\[4pt]
\mathfrak h_{2} &\!\!\!\colon\!\!\!&
\frac{1}{2}[X_1,X_5]=[X_2,X_3]=X_1,\ [X_2,X_5]=X_2,\ [X_3,X_5]=X_3,\
[X_4,X_5]=-3X_4 ;\\[4pt]
\mathfrak h_{3} &\!\!\!\colon\!\!\!&
\frac{1}{2}[X_1,X_4]=[X_2,X_3]=X_1,\
[X_2,X_4]=[X_3,X_5]=X_2,\ [X_2,X_5]=-[X_3,X_4]=-X_3 ;\\[4pt]
\mathfrak h_{4} &\!\!\!\colon\!\!\!& [X_1,X_4]=X_1,\ [X_2,X_5]=X_2,\ [X_3,X_4]=[X_3,X_5]=-X_3 ;\\[4pt]
\mathfrak h_{5} &\!\!\!\colon\!\!\!& [X_1,X_5]=[X_2,X_4]=X_1,\
[X_3,X_4]=X_2,\ [X_3,X_5]=-X_3,\ [X_4,X_5]=X_4.
\end{array}
$$
\end{theorem}

Therefore, all of them are irreducible and $\mathfrak h_{1}$ is the
unique nilpotent Lie algebra having a hypo-contact structure. In
\cite{D} Diatta  gives a list of    solvable contact Lie algebras in
dimension 5 and he shows that, up to isomorphism, there are   three
nilpotent  contact Lie algebras of dimension $5$. By \cite{ConS}
only two of these nilpotent Lie algebras  have hypo structures.
Since the center of the  Lie algebras $\mathfrak
h_2,\ldots,\mathfrak h_5$ is trivial, we conclude that there are
many 5-dimensional solvable contact Lie algebras with no
hypo-contact structures.

In \cite[Theorem 14]{ConS} it is proved that a hypo structure is
Sasakian if and only if it is $\eta$-Einstein \cite{BGM,Ok}. The Lie
algebras described in Theorem~\ref{clasification} cannot be Einstein
because they are solvable  and contact \cite{D}. In Section
\ref{Sasak-eta-Einstein}, we study which of these Lie algebras are
$\eta$-Einstein or, equivalently, Sasakian. We show that the only
5-dimensional solvable Lie algebras admitting a hypo-contact
$\eta$-Einstein structure, are $\mathfrak h_1$ and $\mathfrak h_3$
(Proposition \ref{families-eta-Einstein}). Concerning contact
Calabi-Yau structures recently introduced in \cite{TV}, in
Proposition \ref{clasif-contact-Calabi-Yau} it is proved that there
are no 5-dimensional solvable non-nilpotent Lie algebras admitting
such a structure.

In Section \ref{hypo-vol-eq} we solve the Conti-Salamon evolution
equations for the left-invariant hypo-contact structure on the
simply connected solvable Lie group $H_i$ $(1\leq i \leq 5)$ whose
Lie algebra is $\mathfrak h_i$. More concretely we obtain the
following result.

\begin{theorem}\label{Calabi-Yau-metrics}
Any left-invariant hypo-contact structure on the $5$-dimensional
solvable Lie group $H_i$ $(1\leq i \leq 5)$ gives rise to a metric
with holonomy $SU(3)$ on $H_i\times I$, for some open interval $I$.
\end{theorem}

This theorem is an existence result; in fact, our  metric is
explicit only for the left-invariant hypo-contact structure on the
nilpotent Lie group $H_1$, recovering in this way the well-known
example obtained in \cite{GLPS}.

Finally, Section \ref{half-flat} is devoted to show the existence of
new metrics with holonomy $G_2$. To this end, using Proposition
\ref{from-hypo-to-half-flat}, we consider the circle bundles over
$H_i$ $(1\leq i\leq 5)$ whose total space $K_i$ has a half-flat
structure induced by the left-invariant hypo-contact structure on
$H_i$. Solving the Hitchin evolution equations, we prove the
following theorem.

\begin{theorem}\label{G2-metrics}
The half-flat structure on $K_i$ $(i=1,4, 5)$ gives rise to a metric
with holonomy $G_2$ on $K_i \times I$,  for some open interval $I$.
\end{theorem}

We must notice that the above metric on $K_1 \times I$  agrees with
the one obtained in  \cite{CFino}. However, as far as we know, the
other metrics on $K_i \times I$ $(i=4,5)$ are new and, as we explain
in Section  \ref{half-flat}, they can be considered as a \lq \lq
deformation\rq \rq  of the metric with holonomy $G_2$ found in
\cite{CFino}.

\section{Hypo-contact structures}\label{hypo-halff}
In this section, we study $5$-manifolds with a \emph{hypo-contact}
structure, that is, a hypo structure in the sense of \cite{ConS}
carrying a contact form. We prove that there exists such a structure
on any oriented hypersurface of an special half-flat manifold,
namely, such that the K\"ahler form is preserved by the normal
vector field and its differential is equal two times the real part
of the $(3,0)$-form. First we need to recall some properties of hypo
structures on $5$-manifolds.

Let $N$ be a $5$-manifold with an $SU(2)$-structure
$(\eta,\omega_1,\omega_2,\omega_3)$, that is to say \cite{ConS},
$\eta$ is a nowhere vanishing $1$-form and $\omega_i$ are $2$-forms
on $N$ satisfying
  \begin{equation}\label{wedge-i-j}
  \omega_i\wedge\omega_j=\delta_{ij}v, \quad
  v\wedge\eta\not=0,
  \end{equation}
for some nowhere vanishing $4$-form $v$, and
   \begin{equation}
   i_X\omega_3=i_Y\omega_1\Rightarrow \omega_2(X,Y)\ge 0,
  \end{equation}
where $i_X$ denotes the contraction by $X$.

An $SU(2)$-structure determined by
$(\eta,\omega_1,\omega_2,\omega_3)$ is called \emph{hypo} if the
following equations
  \begin{equation}\label{rhypo}
d\omega_3=0, \qquad d(\eta\wedge\omega_1)=0,
  \qquad d(\eta\wedge\omega_2)=0
  \end{equation}
are  satisfied \cite{ConS}.

\begin{definition}\label{hypo-contact}
We say that an $SU(2)$-structure $(\eta,\omega_1,\omega_2,\omega_3)$
on a manifold $N$ is \emph{ hypo-contact} if it  satisfies
$$
d\eta=-2\omega_3, \qquad
 d(\eta\wedge\omega_1)=0, \qquad
 d(\eta\wedge\omega_2)=0.
$$
\end{definition}

Regarding the intrinsic torsion of these $SU(2)$-structures, we
recall that in Proposition $10$ of \cite{ConS}, it is proved that
the hypo structures are the $SU(2)$-structures whose intrinsic
torsion takes values in the space $2 \mathbb R  \oplus \Lambda^1
(\mathbb R^4)^* \oplus 3 \Lambda^2_- (\mathbb R^4)^*$. Now, one can
check that the hypo-contact structures are the $SU(2)$-structures
whose intrinsic torsion lies in the $SU(2)$-module
$$
2\mathbb R \oplus 2  \Lambda^2_- (\mathbb R^4)^*.
$$

An $SU(2)$-structure  on $N$ induces an $SU(3)$-structure
$(F,\Psi_+,\Psi_-)$ on $N\times\mathbb R$ defined by
  \begin{equation}
 F=\omega_3+\eta\wedge dt, \quad
  \Psi=\Psi_++i\Psi_-=(\omega_1+i\omega_2)\wedge(\eta+idt),
  \end{equation}
where $t$ is a coordinate on $\mathbb R$. Vice versa, let $f:
N\longrightarrow M$ be an oriented hypersurface of a $6$-manifold
$M$ with an $SU(3)$-structure $(F,\Psi_+,\Psi_-)$, and denote by
$\mathbb U$ the unit normal vector field. Then $N$ has an
$SU(2)$-structure $(\eta,\omega_1,\omega_2,\omega_3)$ given by
  \begin{equation}\label{hyp1}
\eta=-i_{\mathbb U}F, \quad
 \omega_3= f^*F,\quad
\omega_1=i_{\mathbb U} \Psi_-, \quad
 \omega_2=-i_{\mathbb U} \Psi_+. \quad
  \end{equation}

If $M$ has holonomy contained in $SU(3)$, that is,  if the
$SU(3)$-structure $(F,\Psi_+,\Psi_-)$ is integrable or,
equivalently,
  $$
  dF=d\Psi_+=d\Psi_-=0,
  $$
any oriented hypersurface $N$ of $M$ is naturally endowed with a
hypo structure \cite{ConS}. Indeed, the conditions
$dF=d\Psi_+=d\Psi_-=0$ imply that the induced $SU(2)$-structure on
$N$ defined by \eqref{hyp1} satisfies \eqref{rhypo}. If in addition
the Lie derivative ${\mathcal L}_{\Bbb U} F $ is equal to $2 f^*
(F)$, then the induced $SU(2)$-structure is hypo-contact.

Concerning the converse, Conti and Salamon \cite{ConS} prove that a
real analytic hypo structure  on  a real analytic $5$-manifold $N$
can be lifted to an integrable $SU(3)$-structure on $N\times I$, for
some open interval $I$.  More precisely, they show that if
$(\eta,\omega_1,\omega_2,\omega_3)$ belongs to a one-parameter
family of hypo structures
$(\eta(t),\omega_1(t),\omega_2(t),\omega_3(t))$ satisfying the
evolution equations
 \begin{equation}\label{vol}
  \left\{\begin{array}{l}
  \partial_t\omega_3=-d\eta\\
  \partial_t(\omega_2\wedge\eta)=d\omega_1\\
   \partial_t(\omega_1\wedge\eta)=-d\omega_2,
  \end{array} \right.
  \end{equation}
  for  all $t$ lying in some open interval $I$, then the $SU(3)$-structure $(F, \Psi_+, \Psi_-)$ on $N \times I$ given by
$$
F = \eta (t) \wedge dt + \omega_3 (t), \quad \Psi = \Psi_+ + i
\Psi_- = (\omega_1 (t) + i \omega_2(t)) \wedge (\eta(t) + i dt))
$$
is integrable.

In Section \ref{hypo-vol-eq} we shall back to the  equations
\eqref{vol}. Now, we weaken the integrability condition of the
$SU(3)$-structure
 $(F, \Psi_+, \Psi_-)$ on $M$ to be
half-flat in the sense of \cite{CS}, that is  $d(F\wedge
F)=d\Psi_+=0$. First we show  how  to lift a hypo structure on a
$5$-manifold $N$ to a hal-flat structure on the total space of a
circle bundle over $N$.

\begin{proposition}\label{from-hypo-to-half-flat}
Let $N$ be a $5$-manifold equipped with a hypo structure
$(\eta,\omega_1,\omega_2,\omega_3)$. For any integral closed 2-form
$\Omega$ on $N$ annihilating both $\omega_3$ and $\cos \theta\,
\omega_1+ \sin\theta\, \omega_2$ for some $\theta$, there is a
principal circle bundle $\pi\colon M\longrightarrow N$ with
connection form $\rho$ such that $\Omega$ is the curvature of $\rho$
and such that the $SU(3)$-structure $(F^\theta, \Psi^\theta_+,
\Psi^\theta_-)$ on $M$ given by
$$
\begin{array}{l}\label{hypo-to-half-flat}
F^\theta= \pi^*(\cos \theta\,\omega_1+ \sin\theta\, \omega_2) +
\pi^*(\eta)\wedge \rho, \\[5pt]
\Psi^\theta_+= \pi^*((-\sin\theta\,\omega_1+\cos
\theta\,\omega_2)\wedge\eta) -  \pi^*(\omega_3)\wedge
\rho, \\[5pt]
\Psi^\theta_-= \pi^*(-\sin\theta\,\omega_1+\cos
\theta\,\omega_2)\wedge \rho + \pi^*(\omega_3)\wedge \pi^*(\eta),
\end{array}
$$
is half-flat.
\end{proposition}

\begin{proof}
Since $d\rho=\pi^*(\Omega)$, a simple calculation shows that
$$
d(F^\theta\wedge F^\theta)= -2\pi^*(\eta)\wedge \pi^*((\cos
\theta\,\omega_1+ \sin\theta\, \omega_2)\wedge \Omega)=0,
$$
and
$$
d(\Psi^\theta_+)= -\pi^*(\omega_3\wedge \Omega)=0.
$$
The existence of a principal circle bundle in the conditions above
follows from a well known result by Kobayashi \cite{Kobayashi}.
\end{proof}

\begin{remark}
Notice that $\Omega=0$ satisfies the hypothesis in the previous
proposition for each $\theta$ and one gets the trivial circle bundle
$M=N\times \mathbb R$ with the half-flat structure which is the
natural extension to $M$ of the hypo structure on $N$. In
Section~\ref{half-flat} we show non-trivial solutions on circle
bundles over solvable Lie groups with a hypo-contact structure.
\end{remark}

 \medskip

As a consequence of Proposition \ref{from-hypo-to-half-flat} we have

\begin{corollary}
Let $N$ be a $5$-manifold with a hypo-contact structure
$(\eta,\omega_1,\omega_2,\omega_3)$. For any $\theta$, let us
consider the half-flat structure on $N\times\mathbb R$ defined in
Proposition \ref{hypo-to-half-flat}. Then,
$$
 i_{\mathbb U}(dF^\theta-2\Psi^\theta_{+})=0,
$$
where $\mathbb U$ denotes the vector field on $\mathbb R$ dual to
$\rho=dt$.
\end{corollary}

\begin{proof}
Clearly $dF^\theta=(\cos \theta\,d\omega_1+ \sin\theta\, d\omega_2)
- 2\omega_3\wedge dt$, since $d\eta=-2\omega_3$. So, $ i_{\mathbb U}
dF^\theta=-2\omega_3=2 i_{\mathbb U} \Psi^\theta_+$, which proves
that $i_{\mathbb U}(dF^\theta-2\Psi^\theta_{+})=0.$
\end{proof}

\begin{theorem}\label{converse-hypoc-half-flat-0}
Let $M$ be a  $6$-dimensional manifold endowed with a half-flat
structure $(F, \Psi_+, \Psi_-)$. Let $f: N\longrightarrow M$ be an
oriented  hypersurface of $M$. Denote the unit normal vector field
by $\mathbb U$.  Suppose that
 \begin{equation}
dF=2\Psi_{+}, \quad \quad {\mathcal L}_{\mathbb U}F=0,
 \end{equation}
where ${\mathcal L}$ denotes the Lie derivative. Then, the forms
$(\eta,\omega_1,\omega_2,\omega_3)$ on $N$ given by
 \begin{equation}\label{hypo-nonc}
  \eta=-i_{\mathbb U}  F,\quad
   \omega_1= -i_{\mathbb U}  \Psi_-, \quad
  \omega_2= f^*F,\quad
  \omega_3=-i_{\mathbb U}  \Psi_+,
  \end{equation}
  define a hypo-contact structure on $N$.
  \end{theorem}

 \begin{proof}
Equations \eqref{hypo-nonc} imply $f^*(\Psi_+) =
-\omega_1\wedge\eta,$ so that $\omega_1\wedge\eta$ is closed if the
$SU(3)$-structure is half-flat. Using again \eqref{hypo-nonc}, we
have
$$
d\eta=-d(i_{\mathbb U} F)=i_{\mathbb U} dF-{\mathcal L}_{\mathbb
U}F= i_{\mathbb U} dF=2i_{\mathbb U}  \Psi_{+}=-2\omega_3,
$$
since ${\mathcal L}_{\mathbb U}F=0$ and $dF=2\Psi_{+}$.

To complete the proof, we notice that
$d\omega_2=f^{*}(dF)=2f^{*}(\Psi_+)=-2\omega_1\wedge\eta.$
Therefore, $d(\omega_2\wedge\eta)=d\omega_2\wedge\eta+\omega_2\wedge
d\eta=0$.
 \end{proof}

\medskip

An example of a $6$-manifold satisfying the conditions of the
Theorem \ref{converse-hypoc-half-flat-0} is the compact nilmanifold
defined by the equations
 $$
 de^i=0   \quad  (1\leq i\leq 4),  \quad \quad
 d e^5=-2e^{14}-2e^{23}, \quad \quad     d e^6=-2e^{13}+2e^{24},
 $$
with the half-flat structure $(F, \Psi_+, \Psi_-)$ given by
$$
F=e^{12} + e^{34} + e^{56},  \quad \quad \Psi_+= e^{135} - e^{146} -
e^{236} - e^{245},    \quad \quad \Psi_-= e^{136} + e^{145} +
e^{235} - e^{246}.
$$
Consider the $5$-submanifold whose unit normal vector field is the
dual to $-e^6$, that is, the $5$-dimensional compact submanifold
determined by the equations 
$de^i=0   \  (1\leq i\leq 4)$, $d e^5=-2e^{14}-2e^{23}$.
Then, the equations $dF=2\Psi_{+}$ and ${\mathcal L}_{\mathbb U}F=0$
are satisfied.

\begin{proposition}\label{converse-hypoc-half-flat}
Let $M$ be a  $6$-dimensional manifold endowed with a half-flat
structure $(F, \Psi_+, \Psi_-)$, and let $f: N\longrightarrow M$ be
an oriented  hypersurface
  of $M$. Denote the unit normal vector field by $\mathbb U$.  Suppose that
 \begin{equation}
   g(\nabla_{\mathbb U}{\mathbb U},X)=0, \quad \quad {\cal L}_{\mathbb U}\Psi_+=0,
 \end{equation}
   for any vector field $X$ on $N$.
 Then,
the forms  $(\eta,\omega_1,\omega_2,\omega_3)$ on $N$ given by
 \eqref{hypo-nonc}
 define a hypo structure on $N$.
  \end{proposition}

  \begin{proof}
 Proceeding as in Theorem \ref{converse-hypoc-half-flat-0}
 we see that $d(\omega_1\wedge\eta)=0$. Moreover,
 taking account \eqref{hypo-nonc},
 we have $d \omega_3=-d(i_{\mathbb U} \Psi_+)=
 i_{\mathbb U} d\Psi_{+} - {\cal L}_{\mathbb U}\Psi_+=0$
because both terms vanish. 
Therefore, only it remains to prove that 
$d(\omega_2\wedge\eta)=0$.

Denote by $\rho$ the $1$-form on $M$ dual to the normal vector field
$\mathbb U$, and by ${\mathfrak X}(M)$ the Lie algebra of the vector
fields on $M$. Then, the restriction ${\mathfrak X}(M)_{\mid N}$ to
$N$ of ${\mathfrak X}(M)$ is the direct sum
$$
{\mathfrak X}(M)_{\mid N} = {\mathfrak X}(N) \oplus \mathbb U.
$$
Firstly, we see that, for any
vector fields $X$, $Y$ on $N$, $d\rho(X,Y)=d\rho({\mathbb U},X)=0$.
In fact, we  have
\begin{equation}\label{equat-1}
d\rho(X,Y)=X\rho(Y) - Y\rho(X)- \rho[X,Y] =0.
\end{equation}
Also, for any vector field $X$ on $N$, we get
\begin{equation}\label{equat-2}
d\rho({\mathbb U},X)={\mathbb U}\rho(X) - X\rho({\mathbb U})-
\rho[{\mathbb U},X] =- \rho[{\mathbb U},X] =0,
\end{equation}
since the normal component of $[{\mathbb U},X]$ is
$$
g({\mathbb U},[{\mathbb U},X]) = g({\mathbb U},\nabla_{\mathbb
U}X-\nabla_{X}{\mathbb U}) = g({\mathbb U},\nabla_{\mathbb U}X) =
g(\nabla_{\mathbb U}{\mathbb U},X)=0.
$$
From equations \eqref{hypo-nonc} it follows that $F=\omega_2 + \eta
\wedge \rho$. Now from \eqref{equat-1}, \eqref{equat-2} and using
that $\omega_2\wedge d \omega_2=0$, we get
$$
0=d(F\wedge F)=2(\omega_2\wedge d \omega_2+d(\omega_2\wedge \eta)
\wedge \rho -\omega_2\wedge \eta \wedge d\rho) =2d(\omega_2\wedge
\eta) \wedge \rho,
$$
which implies that $d(\omega_2\wedge \eta)=0$.
  \end{proof}

\medskip

To finish this section, we consider $SU(3)$-structures on a manifold
with a Killing vector field $X$ preserving the $SU(3)$-structure,
and we study the conditions under which the $SU(3)$-structure
induces a hypo-contact structure $(\eta, \omega_i)$ on the
$5$-submanifold $N$ determined by $X$ as follows. Let $M$ be a
$6$-dimensonal manifold endowed with an $SU(3)$-structure
$(F,\Psi_+,\Psi_-)$, and let $X\in {\frak X}(M)$ be a Killing vector
field on $M$ which preserves the $SU(3)$-structure, that is $X$ is
an infinitesimal isometry satisfying
$$
{\mathcal L}_X F=0,\quad {\mathcal L}_X \Psi_+=0,\quad {\mathcal
L}_X \Psi_-=0.
$$

In a suitable neighborhood of any point $p$ of $M$ where
$X_p\not=0$, let us denote by $N$ the 5-dimensional manifold formed
from the orbits of $X$.

Let $x$ be the function given by
 \begin{equation}\label{function-x}
 x=g(X,X)^{1/2},
 \end{equation}
where $g$ denotes the Riemannian metric on $M$ determined by the
$SU(3)$-structure. Since $X$ is a Killing vector field, we have that
${\mathcal L}_X(x)=0$, so the function $x$ descends to a function on
$N$ which we denote again by $x$.

On the other hand, let us define a 1-form $\alpha$ on $M$ by
 \begin{equation}\label{form-alpha}
\alpha(Z)={1\over x^2}\, g(Z, X),
 \end{equation}
for any $Z\in{\frak X}(M)$. Observe that $\alpha(X)\equiv 1$. The
form $\alpha$ is also invariant by $X$; in fact, since ${\cal
L}_X\alpha=i_X d\alpha+ di_X\alpha$, it suffices to see that
$(i_Xd\alpha)(Z)=0$ for any vector field $Z\in{\frak X}(M)$. But
$$
\begin{array}{rcl}
(i_Xd\alpha)(Z)&=&\displaystyle
d\alpha(X,Z)=X(\alpha(Z))-\alpha([X,Z])={\mathcal L}_X\left({1\over
x^2}\,g(Z,X)\right)- {1\over x^2}\,g({\mathcal L}_X Z, X)
\\[5pt]
&=&\displaystyle {\mathcal L}_X \left({1\over x^2}\right) \,g(Z,X) +
{1\over x^2}\, {\mathcal L}_X(g(Z,X))- {1\over x^2}\,g({\mathcal
L}_X Z, X) = 0,
\end{array}
$$
because $X$ is Killing and $dx(X)=0$. Therefore, $\alpha$ descends
to a 1-form on $N$ which again we denote by the same letter.

\begin{lemma}\label{lemma1}
In the above conditions, the quadruplet of differential forms
$(\eta,\omega_1,\omega_2,\omega_3)$ given by
\begin{equation}\label{induced-SU(2)-st}
\eta=-i_XF,\quad \omega_1=x\, i_X(F\wedge\alpha),\quad \omega_2=i_X
\Psi_-,\quad \omega_3=-i_X \Psi_+,
\end{equation}
defines an $SU(2)$-structure on $N$, where $x$ and $\alpha$ are the
function and the $1$-form on $N$ induced by \eqref{function-x} and
\eqref{form-alpha}, respectively.
\end{lemma}

\begin{proof}
First we show that the Lie derivative of the forms $i_XF, x\,
i_X(F\wedge\alpha), i_X \Psi_-$ and $i_X \Psi_+$ with respect to $X$
is zero, so these forms descend to forms on $N$. In fact, since $X$
preserves the $SU(3)$-structure we have
$$\begin{array}{l}
{\mathcal L}_X(i_XF)=i_X(di_XF) =i_X({\mathcal L}_XF)=0,\\
{\mathcal L}_X(x\, i_X(F\wedge\alpha))=({\mathcal L}_X x)\,
i_X(F\wedge\alpha) + x( {\mathcal L}_X i_X(F\wedge\alpha))=0,\\
{\mathcal L}_X(i_X\Psi_{\pm})=i_X(di_X\Psi_{\pm})=i_X({\mathcal L}_X
\Psi_{\pm})=0.
\end{array}
$$
Now it remains to see that $(\eta,\omega_i)$ defines an
$SU(2)$-structure. Let $E_6={1\over x}X$ be the unitary vector field
in the direction of $X$. We can consider a local orthonormal basis
$E_1,\ldots,E_6$ such that the $SU(3)$-structure expresses in terms
of the dual basis $e^1,\ldots,e^6$ as follows
$$
F=e^{12}+e^{34}+e^{56},\quad
\Psi_+=(e^{13}+e^{42})e^5-(e^{14}+e^{23})e^6,\quad
\Psi_-=(e^{14}+e^{23})e^5+(e^{13}+e^{42})e^6.
$$
Notice that $\alpha={1\over x} e^6$. Therefore, locally we have
$$\begin{array}{l}
\eta=-i_XF=-i_{xE_6}(e^{12}+e^{34}+e^{56})=xe^5,\\
\omega_1=x\, i_X(F\wedge\alpha)=
i_{xE_6}(e^{126}+e^{346})=x(e^{12}+e^{34}),\\
\omega_2=i_X(\Psi_-)=
i_{xE_6}((e^{14}+e^{23})e^5+(e^{13}+e^{42})e^6)=x(e^{13}+e^{42}) ,\\
\omega_3=-i_X(\Psi_+)=
-i_{xE_6}((e^{13}+e^{42})e^5-(e^{14}+e^{23})e^6)=x(e^{14}+e^{23}) ,
\end{array}
$$

\smallskip

\noindent and thus $(\eta,\omega_i)$ is an $SU(2)$-structure.
\end{proof}

\bigskip

\begin{theorem}
Let $M$ be a $6$-manifold in the conditions of Lemma \ref{lemma1}.
Suppose that $X$ is a Killing vector field of constant lenght,
preserving a half-flat $SU(3)$-structure on $M$ and satisfying
 $d\alpha\wedge i_X\Psi_+=0$. Then the
 structure on $N$ given by
\eqref{induced-SU(2)-st} is hypo. If in addition
 $i_X(dF-2\Psi_+)=0$, the $SU(2)$-structure
 on $N$ is hypo-contact.
 \end{theorem}

\begin{proof}
First we notice that for any $SU(3)$-structure on $M$, the
SU(2)-structure on $N$ defined by \eqref{induced-SU(2)-st} satisfies
$$
\omega_1\wedge\eta = -{x\over 2} \, i_X(F\wedge F),\qquad
\omega_2\wedge\eta = x^2 \, i_X(\alpha\wedge \Psi_+) = x^2(\Psi_+ -
\alpha\wedge i_X\Psi_+).
$$

\smallskip
\noindent Therefore, we get
$$
\begin{array}{l}
-2 d(\omega_1\wedge\eta)= dx\wedge i_X(F\wedge F) - x\,
i_Xd(F\wedge F),\\
d(\omega_2\wedge\eta)= 2x dx\wedge(\Psi_+ - \alpha\wedge i_X\Psi_+)
+ x^2 \bigl(  d\Psi_+ - d\alpha\wedge i_X\Psi_+ -\alpha\wedge
i_X(d\Psi_+) \bigr),\\
d\omega_3=i_X(d\Psi_+),\quad d\eta+2\omega_3=i_X(dF-2\Psi_+).
\end{array}
$$
Now, let us consider a Killing vector field $X$ of constant lenght
such that it preserves a half-flat $SU(3)$-structure
$(F,\Psi_+,\Psi_-)$ on $M$ and satisfies $d\alpha\wedge
i_X\Psi_+=0$, then the structure on $N$ given by
\eqref{induced-SU(2)-st} is hypo since
$$
d\omega_3 =0,\quad\quad d(\omega_1\wedge\eta)=0,\quad\quad
d(\omega_2\wedge\eta)= - x^2 \, d\alpha\wedge i_X\Psi_+=0.
$$
Moreover, if  $i_X(dF-2\Psi_+)=0$, then $d\eta=-2\omega_3$, and so
the $SU(2)$-structure on $N$ is hypo-contact.
\end{proof}

\medskip

The previous study is done in the same vein of the papers \cite{AS}
and \cite{ConTom} where  $S^1$-bundles with a  $U(1)$-invariant
$SU(3)$-structure (or $G_2$-structure) are considered.

\begin{remark}
We must notice that in the conditions of Lemma~\ref{lemma1}, if $X$
is a Killing vector field on $M$ preserving the $SU(3)$-structure
(not necessarily half-flat) and satisfying $i_X(dF-2\Psi_+)=0$, then
the 1-form $\eta$ is a contact form on $N$.
\end{remark}

\section{Solvable Lie algebras with a hypo-contact structure}\label{hypo-contact-solvable}

The purpose of this Section is to prove Theorem~\ref{clasification}.
First, we need to show the following propositions.

\begin{proposition}\label{red-equations}
Let $\mathfrak g$ be a solvable Lie algebra of dimension $5$ with a
hypo-contact structure $(\eta, \omega_1,\omega_2, $ $\omega_3)$.
Then, there is a basis $e^1,\ldots,e^5$ for ${\frak g}^*$ such that
\begin{equation}\label{structure}
\eta=e^5,\quad\quad \omega_1= e^{12} + e^{34},\quad\quad \omega_2=
e^{13} + e^{42},\quad\quad \omega_3= e^{14} + e^{23},
\end{equation}
and
\begin{equation}\label{differentials}
\left\{
\begin{array}{rcl}
d e^1 \!\!&\!\!=\!\!&\!\! A e^{14} + A e^{23},\\[7pt]
d e^2 \!\!&\!\!=\!\!&\!\! B_{12} e^{12} + B_{13} e^{13} + B_{14}
e^{14} +
B_{15} e^{15} - B_{14} e^{23} + (2 A + B_{13}) e^{24}\\[5pt]
&&+ B_{25} e^{25} + B_{34}
e^{34} + B_{35} e^{35},\\[7pt]
d e^3 \!\!&\!\!=\!\!&\!\! (3 A + B_{13}) e^{12} + C_{13} e^{13} +
C_{14} e^{14} +
C_{15} e^{15} - C_{14} e^{23} - (B_{12} + B_{34}-C_{13}) e^{24}\\[5pt]
&&+ C_{25} e^{25} - (A + B_{13})
e^{34} - B_{25} e^{35},\\[7pt]
d e^4 \!\!&\!\!=\!\!&\!\! B_{14} e^{12} + C_{14} e^{13} + (B_{34} -
C_{13}) e^{14} +
D_{15} e^{15} + (B_{12} + C_{13}) e^{23} + C_{14} e^{24}\\[5pt]
&&+ C_{15} e^{25} - B_{14}
e^{34} - B_{15} e^{35},\\[7pt]
d e^5 \!\!&\!\!=\!\!&\!\!  -2 e^{14} - 2 e^{23},
\end{array}
\right.
\end{equation}
where the coefficients $A, B_{12}, B_{13}, B_{14}, B_{15}, B_{25},
B_{34}, B_{35}, C_{13}, C_{14}, C_{15}, C_{25}$ and $D_{15}$ satisfy
the conditions
\begin{equation}\label{jacobi}
d(de^i)=0
\end{equation}
for $i=1,2,3,4,5$.
\end{proposition}

\begin{proof}
Let $V$ be the subspace of $\mathfrak g^*$ orthogonal to $\eta$.
Since $\mathfrak g$ is solvable, there is a nonzero element
$\alpha\in \mathfrak g^*$ which is closed. Thus,
$$
\alpha = \beta + \lambda\, \eta,
$$
where $\beta\in V$ and $\lambda\in \mathbb{R}$. Now, $d\alpha=0$ is
equivalent to $d\beta = - \lambda\, d\eta$. Therefore,
$\gamma={1\over \Vert\beta\Vert}\, \beta$ is a unit element in
$V=\langle \eta \rangle^\perp$ satisfying
$$
d\gamma = \tau\, d\eta,
$$
with $\tau =-\lambda/\Vert\beta\Vert$. From \cite[Corollary
3]{ConS}, there is a basis $e^1,\ldots,e^5$ for ${\mathfrak g}^*$
satisfying \eqref{structure} with $e^1=\gamma$.

Therefore, $de^5=d\eta=-2\omega_3=-2 e^{14} - 2 e^{23}$ and
$de^1=d\gamma= \tau\, d\eta= A\, e^{14} +A\, e^{23}$, where
$A=-2\tau$, so the differentials of $e^1,\ldots,e^5$ are given by
\begin{equation}\label{gen-differentials}
\left\{
\begin{array}{rl}
d e^1 =\!\!&\!\! A\, e^{14} + A\, e^{23},\\[5pt]
d e^2 =\!\!&\!\! B_{12} e^{12} + B_{13} e^{13} + B_{14} e^{14} +
\cdots\cdots + B_{34} e^{34}+ B_{35} e^{35}+
B_{45} e^{45},\\[5pt]
d e^3 =\!\!&\!\! C_{12} e^{12} + C_{13} e^{13} + C_{14} e^{14} +
\cdots\cdots + C_{34} e^{34}+ C_{35} e^{35}+
C_{45} e^{45},\\[5pt]
d e^4 =\!\!&\!\! D_{12} e^{12} + D_{13} e^{13} + D_{14} e^{14} +
\cdots\cdots +
D_{34} e^{34}+ D_{35} e^{35}+ D_{45} e^{45},\\[5pt]
d e^5 =\!\!&\!\! -2 e^{14} - 2 e^{23},
\end{array}
\right.
\end{equation}
where the coefficients must satisfy the Jacobi identity $d(de^i)=0$,
$1\leq i\leq 5$, and the additional conditions
$d(\eta\wedge\omega_1)=d(\eta\wedge\omega_2)=0$ in order to have a
hypo-contact structure. By imposing  that  $d(e^{125}+e^{345}) =
d(\eta\wedge\omega_1)=0$, $d(e^{135}-e^{245}) =
d(\eta\wedge\omega_2)=0$ and $d(e^{14}+e^{23})=-(1/2) d(d e^5)=0$,
the coefficients in \eqref{gen-differentials} satisfy the following
relations:
%
$$
\begin{array}{rl}
B_{23} \!\!\!&\!\!\! = -B_{14}, \   B_{24}= 2A + B_{13},\  B_{45}=
0,\
C_{12}= 3 A + B_{13},\    C_{23}= - C_{14},\  C_{24} = -B_{12}-B_{34} +C_{13}, \\[5pt]
 C_{34} \!\!\!&\!\!\! = - A - B_{13},\ C_{35}= -
B_{25},\  C_{45} = 0,\ D_{12} =  B_{14},\  D_{13}= C_{14},\ D_{14}=
B_{34}-
C_{13}, \\[5pt]
D_{23} \!\!\!&\!\!\!= B_{12}+ C_{13},\  D_{24}=  C_{14},\ D_{25}=
C_{15},\ D_{34} = - B_{14},\  D_{35}= - B_{15},\  D_{45}= 0.
\end{array}
$$
This completes the proof of \eqref{differentials}. Notice that the
coefficients must also satisfy \eqref{jacobi}.
\end{proof}

Let $E_1,\ldots, E_5$ be the basis for $\mathfrak g$ dual to the
basis $e^1,\ldots,e^5$ and let us denote by $c_{ijk}^l$ the
component in $E_l$ of $\bigl[[E_i,E_j],E_k\bigr]+
\bigl[[E_j,E_k],E_i\bigr]+ \bigl[[E_k,E_i],E_j\bigr]$. It is clear
that the Jacobi identity is satisfied if and only if $c_{ijk}^l=0$
for $1\leq i<j<k\leq 5$ and $1\leq l\leq 5$.

A direct calculation shows that $c_{134}^4=c_{134}^3=c_{134}^2=0$ if
and only if
\begin{equation}\label{B15-B25-B35}
\begin{array}{rl}
2 B_{15} =&  B_{12} B_{14} + B_{14} B_{34} + 2 B_{14} C_{13} - 2 B_{13} C_{14},\\[6pt]
2 B_{25}=& B_{12} B_{13} + 4 A B_{34} + 3 B_{13} B_{34}
- 2 B_{13} C_{13} - 2 B_{14} C_{14},\\[6pt]
2 B_{35} =& 2 A B_{13} + 2 B_{13}^2 + 2 B_{14}^2 - B_{12} B_{34} +
B_{34}^2,
\end{array}
\end{equation}
respectively. Moreover, $c_{123}^3=c_{124}^3=c_{123}^4=0$ if and
only if
\begin{equation}\label{C15-C25-D15}
\begin{array}{l}
2 C_{15}\! =\!- 4 A B_{14} - 2 B_{13} B_{14} + 3 B_{12} C_{14} + B_{34} C_{14},\\[6pt]
2 C_{25}\! =\!-12 A^2 \!-\! B_{12}^2\!-\! 10 A B_{13} \!-\! 2
B_{13}^2 - 2 B_{12}
B_{34} - B_{34}^2 + 3 B_{12} C_{13} + 3 B_{34} C_{13} - 2 C_{13}^2 - 2 C_{14}^2,\\[6pt]
2 D_{15}\! =\!- B_{12}^2 - 2 B_{14}^2 + B_{12} B_{34} - 3 B_{12}
C_{13} + B_{34} C_{13} - 2 C_{13}^2 - 2 C_{14}^2,
\end{array}
\end{equation}
respectively.

\begin{corollary}\label{red-equations-2}
Let $\mathfrak g$ be a solvable Lie algebra with a hypo-contact
structure $(\eta, \omega_1,\omega_2,\omega_3)$. Then, there is a
basis $e^1,\ldots,e^5$ of ${\mathfrak g}^*$ satisfying
\eqref{structure}, \eqref{differentials}, \eqref{B15-B25-B35},
\eqref{C15-C25-D15} and where the seven remaining coefficients $A,
B_{12}, B_{13}, B_{14}, B_{34}, C_{13}, C_{14}$ satisfy the Jacobi
identity \eqref{jacobi}.
\end{corollary}

\begin{proposition}\label{8families}
Let $\mathfrak g$ be a solvable Lie algebra with a basis
$e^1,\ldots,e^5$ for $\mathfrak g^*$ in the conditions of
Proposition~\ref{red-equations}. Then, the structure equations
\eqref{differentials} reduce to one of the following six families:
\begin{equation}\label{Azero-ecus-1}
\left\{
\begin{array}{l}
d e^1= 0,\\[4pt]
d e^2= r e^{12},\\[4pt]
d e^3=  r e^{13},\\[4pt]
d e^4=  - r e^{14} - 3 r^2  e^{15} + 2 r e^{23},\\[4pt]
d e^5 = -2 e^{14} - 2 e^{23},
\end{array}
\right.
\end{equation}
where $r\in \mathbb{R}^*$; moreover, $\mathfrak g^1=\langle
E_2,E_3,E_4,E_5 \rangle$, $\mathfrak g^2=\langle r E_4-E_5 \rangle$
and $\mathfrak g^3=0$.
\begin{equation}\label{Azero-ecus-2}
\left\{
\begin{array}{l}
d e^1= 0,\\[4pt]
d e^2= r e^{12} + 3 r e^{34} + 3 r^2 e^{35},\\[4pt]
d e^3=  r e^{13} - 3 r e^{24} - 3 r^2 e^{25} ,\\[4pt]
d e^4= -r d e^5
,\\[4pt]
d e^5 = -2 e^{14} - 2 e^{23},
\end{array}
\right.
\end{equation}
where $r\in \mathbb{R}^*$; moreover, $\mathfrak g^1=\langle E_2,E_3,
r E_4-E_5 \rangle$, $\mathfrak g^2=\langle r E_4-E_5 \rangle$ and
$\mathfrak g^3=0$.
\begin{equation}\label{Azero-ecus-3}
\left\{
\begin{array}{l}
d e^1= 0,\\[4pt]
d e^2= r e^{14} - r e^{23} - a r e^{25} + r^2 e^{35},\\[4pt]
d e^3= {a\over r} de^2
,\\[4pt]
d e^4= r e^{12} + a e^{13}  - (a^2+r^2) e^{15} + a e^{24} - r e^{34},\\[4pt]
d e^5 = -2 e^{14} - 2 e^{23},
\end{array}
\right.
\end{equation}
where $a\in \mathbb{R}$ and $r\in \mathbb{R}^*$; moreover,
$\mathfrak g^1=\langle r E_2+a E_3,E_4,E_5 \rangle$ and $\mathfrak
g^2=0$.
\begin{equation}\label{Azero-ecus-4}
\left\{
\begin{array}{l}
d e^1= d e^2= 0,\\[4pt]
d e^3= a e^{13} + b e^{14} - b e^{23} + a e^{24} - (a^2+b^2) e^{25},\\[4pt]
d e^4= b e^{13} - a e^{14} - (a^2+b^2) e^{15} + a e^{23} + b e^{24},\\[4pt]
d e^5 = -2 e^{14} - 2 e^{23},
\end{array}
\right.
\end{equation}
where $a,b\in \mathbb{R}$; moreover, if $a$ or $b$ is nonzero then
$\mathfrak g^1=\langle E_3,E_4,E_5 \rangle$ and $\mathfrak g^2=0$.
\begin{equation}\label{Azero-ecus-5}
\left\{
\begin{array}{l}
d e^1= 0,\\[4pt]
d e^2=  r e^{34} + {r^2 \over 2}  e^{35},\\[4pt]
d e^3=  r e^{13},\\[4pt]
d e^4= - {r^2\over 2} e^{15} + r e^{23},\\[4pt]
d e^5 = -2 e^{14} - 2 e^{23},
\end{array}
\right.
\end{equation}
where $r\in \mathbb{R}^*$; moreover, $\mathfrak g^1=\langle
E_2,E_3,E_4,E_5 \rangle$, $\mathfrak g^2=\langle E_2, r E_4-2E_5
\rangle$ and $\mathfrak g^3=0$.
\begin{equation}\label{Azero-ecus-7}
\left\{
\begin{array}{l}
d e^1= 0,\\[4pt]
d e^2= r e^{12} + a e^{13} + a e^{24}+ {a\over 2r}(r^2 + a^2) e^{25}
+
{a^2\over r} e^{34} + {a^2\over 2r^2}(r^2 + a^2) e^{35},\\[4pt]
d e^3= a e^{12} + {a^2 \over r} e^{13} - r e^{24} - {1\over 2}(r^2 +
a^2) e^{25} - a
e^{34} - {a\over 2r}(r^2 + a^2) e^{35},\\[4pt]
d e^4= - {(r^2+a^2)^2\over 2r^2} e^{15} + {r^2 + a^2 \over r} e^{23},\\[4pt]
d e^5 = -2 e^{14} - 2 e^{23},
\end{array}
\right.
\end{equation}
where $r\in \mathbb{R}^*$ and $a\in \mathbb{R}$; moreover,
$\mathfrak g^1=\langle E_2,E_3,E_4,E_5 \rangle$, $\mathfrak
g^2=\langle a E_2-r E_3, (r^2+a^2)E_4-2 r E_5 \rangle$ and
$\mathfrak g^3=0$.
\end{proposition}

\begin{proof}
We divide the proof in two cases: $A=0$ and $A\not=0$.

Let us suppose first that $A=0$. By Corollary~\ref{red-equations-2}
the coefficients $B_{15}, B_{25}, B_{35}, C_{15}, C_{25}, D_{15}$
are determined by $B_{12}, B_{13}, B_{14}, B_{34}, C_{13}, C_{14}$,
and they are explicitly given by \eqref{B15-B25-B35} and
\eqref{C15-C25-D15}. A direct calculation shows that
$$
\begin{array}{rl}
& c_{125}^4 =0 \quad \Leftrightarrow\quad B_{13} (-3 B_{12} - B_{34}
-
2 C_{13}) B_{14} + (3 B_{12}^2 + 2 B_{13}^2 + B_{12} B_{34}) C_{14} =0;\\[7pt]
& c_{135}^4=0 \quad\Leftrightarrow\quad (-2 B_{13}^2 - B_{12} C_{13}
- B_{34} C_{13} - 2 C_{13}^2) B_{14}
+ B_{13} (3 B_{12} + B_{34} + 2 C_{13}) C_{14} =0;\\[7pt]
& c_{145}^2=0 \quad\Leftrightarrow\quad B_{13} (B_{34} - C_{13}) B_{14} + (B_{13}^2 - B_{12} B_{34}) C_{14} =0;\\[7pt]
& c_{145}^3=0 \quad\Leftrightarrow\quad (B_{13}^2 - B_{12} C_{13} -
B_{34} C_{13} + C_{13}^2) B_{14} + B_{13} (B_{34} - C_{13}) C_{14}
=0.
\end{array}
$$

Let us denote by $\rho_{ij}$ the determinant of the system given by
the equations $i$ and $j$ above, i.e.
$$
\begin{array}{rl}
& \rho_{12}= (B_{13}^2 - B_{12} C_{13}) (-3 B_{12}^2 + 4 B_{13}^2 -
4 B_{12} B_{34} -
B_{34}^2 - 6 B_{12} C_{13} - 2 B_{34} C_{13}),\\[7pt]
& \rho_{13}= -3 B_{13} (B_{12} + B_{34}) (B_{13}^2 - B_{12} C_{13}),\\[7pt]
& \rho_{14}= -(B_{13}^2 - B_{12} C_{13}) (3 B_{12}^2 + 2 B_{13}^2 +
4 B_{12} B_{34} + B_{34}^2 - 3 B_{12} C_{13} - B_{34} C_{13}),\\[7pt]
& \rho_{23}= -(B_{13}^2 - B_{12} C_{13}) (2 B_{13}^2 + B_{12} B_{34}
+ B_{34}^2 + 2 B_{34}
C_{13}),\\[7pt]
& \rho_{24}= - 3 B_{13}(B_{12} + B_{34}) (B_{13}^2 - B_{12} C_{13}) = \rho_{13},\\[7pt]
& \rho_{34}= - (B_{13}^2 - B_{12} C_{13}) (B_{13}^2 - B_{12} B_{34}
- B_{34}^2 + B_{34} C_{13}).
\end{array}
$$

\noindent {\bf Case 1:} At least one of the determinants $\rho_{ij}$
is nonzero. In this case, $B_{14}=C_{14}=0$. Moreover,
$$c_{125}^2= B_{13}(2 B_{13}^2 +
B_{12} B_{34} + B_{34}^2 - 3 B_{12} C_{13} - B_{34} C_{13}),$$ and
$$c_{245}^3= -2 B_{13}(B_{13}^2 -
B_{12} B_{34} - B_{34}^2 + B_{34} C_{13}),$$ which implies that
$B_{13}(B_{13}^2 - B_{12} C_{13})=0$ in order to the Jacobi identity
be satisfied. Since $\rho_{ij}\not=0$ for some $i,j$, it is
necessary that $B_{13}=0$.

Since $B_{13}=B_{14}=C_{14}=0$ the equations \eqref{differentials}
reduce to
\begin{equation}\label{differentials-case1}
\left\{
\begin{array}{l}
d e^1= 0,\\[7pt]
d e^2= B_{12} e^{12} + B_{34}
e^{34} + {1\over 2}B_{34}(B_{34}-B_{12}) e^{35},\\[7pt]
d e^3=  C_{13} e^{13}  - (B_{12} + B_{34}-C_{13}) e^{24} -{1\over
2}(B_{12} + B_{34} - 2 C_{13}) (B_{12} + B_{34} -
C_{13}) e^{25} ,\\[7pt]
d e^4=  (B_{34} - C_{13}) e^{14} -{1\over 2}(B_{12} + C_{13})
(B_{12} - B_{34} + 2 C_{13}) e^{15} +
(B_{12} + C_{13}) e^{23} ,\\[7pt]
 d e^5 = -2 e^{14} - 2 e^{23},
\end{array}
\right.
\end{equation}
because $B_{15}=B_{25}=C_{15}=0$, $B_{35}={1\over 2}B_{34} (B_{34} -
B_{12})$, $C_{25}=-{1\over 2}(B_{12} + B_{34} - 2 C_{13}) (B_{12} +
B_{34} - C_{13})$ and $D_{15}=-{1\over 2}(B_{12} + C_{13}) (B_{12} -
B_{34} + 2 C_{13})$. Now, the Jacobi identity is satisfied if and
only if $B_{12} (B_{34} - 3 C_{13}) (B_{12} + B_{34} - C_{13})=0$,
$B_{34} C_{13} (2 B_{12} - B_{34} + C_{13})=0$ and $B_{34} (B_{12} -
C_{13}) (B_{12} + B_{34} - C_{13})=0$. But, the nonvanishing of some
$\rho_{ij}$ implies that $B_{12}$ and $C_{13}$ cannot be zero, so
\begin{equation}\label{case1}
\begin{array}{c}
 (B_{34} - 3 C_{13}) (B_{12} + B_{34} - C_{13})=0, \quad
 B_{34} (2 B_{12} - B_{34} + C_{13})=0,\\[5pt]
 B_{34} (B_{12} - C_{13})(B_{12} + B_{34} - C_{13})=0.
\end{array}
\end{equation}
If $B_{34}=0$ then \eqref{case1} implies $C_{13}=B_{12}\not=0$ and
from~\eqref{differentials-case1} we get~\eqref{Azero-ecus-1} with
$r=B_{12} \in \mathbb{R}^*$. Otherwise, $B_{34}\not=0$ implies that
$B_{34}= 3 B_{12}$ and $C_{13}= B_{12}$, and
equations~\eqref{differentials-case1} reduce to~\eqref{Azero-ecus-2}
with $r=B_{12} \in \mathbb{R}^*$.

\medskip

\noindent {\bf Case 2:} All the determinants $\rho_{ij}$ vanish.
First, we prove that $B_{13}^2 = B_{12} C_{13}$.

In fact, if $B_{13}^2 \not= B_{12} C_{13}$ then all the determinants
$\rho_{ij}$ vanish if and only if $B_{13} (B_{12} + B_{34}) =0$ and
\begin{equation}\label{ros21}
\begin{array}{rl}
& 3 B_{12}^2 - 4 B_{13}^2 + 4 B_{12} B_{34} +
B_{34}^2 + 6 B_{12} C_{13} + 2 B_{34} C_{13} =0,\\[7pt]
& 3 B_{12}^2 + 2 B_{13}^2 + 4 B_{12} B_{34} + B_{34}^2 - 3 B_{12}
C_{13} - B_{34} C_{13} =0,\\[7pt]
&  2 B_{13}^2 + B_{12} B_{34} + B_{34}^2 + 2 B_{34} C_{13} =0,\quad
 B_{13}^2 - B_{12} B_{34} - B_{34}^2 + B_{34} C_{13}=0.
\end{array}
\end{equation}
Notice that $B_{13}$ must be zero, because otherwise
$B_{34}=-B_{12}$ and the equations~\eqref{ros21} would reduce to
$B_{13}^2 - B_{12} C_{13} =0$, contradicting our assumption. Since
$B_{13}=0$ we have that $B_{12},C_{13}\not=0$ and~\eqref{ros21}
become
$$
\begin{array}{rl}
& 3 B_{12}^2 + 6 B_{12} C_{13} + B_{34}(4 B_{12} + B_{34} + 2 C_{13}) =0,\\[7pt]
& 3 B_{12}^2 - 3 B_{12} C_{13} + B_{34}(4 B_{12} + B_{34} - C_{13})
=0,\\[7pt]
&  B_{34}(B_{12} + B_{34} + 2  C_{13}) =0,\quad
 B_{34} (B_{12} +  B_{34} - C_{13})=0.
\end{array}
$$
From the last two equations we have that $B_{34}=0$, because
$C_{13}$ is nonzero. But in such case the first two equations are
satisfied if and only if $B_{12} C_{13}=0$, which is again a
contradiction. Therefore, we conclude that there are no solutions if
$B_{13}^2 \not= B_{12} C_{13}$.

\medskip

\noindent{\bf Case 2.1:} If $B_{13}^2 = B_{12} C_{13}$ and
$B_{12}=0$, then $B_{13}=0$ and the Jacobi identity is satisfied if
and only if
$$
B_{14}B_{34}=B_{14}C_{13}=B_{34}(B_{34}C_{13}-C_{13}^2-C_{14}^2)=0.
$$
So, if $B_{14}\not=0$ then $B_{34}=C_{13}=0$, and the equations
\eqref{differentials} reduce to~\eqref{Azero-ecus-3} with
$a=C_{14}\in \mathbb{R}$ and $r=B_{14}\in \mathbb{R}^*$.

\noindent On the other hand, if $B_{14}=B_{34}=0$ then we obtain
equations \eqref{Azero-ecus-4} with $a=C_{13}$ and $b=C_{14}\in
\mathbb{R}$.

Finally, let us suppose that $B_{14}=0$ and $B_{34}\not=0$, which
implies that $C_{13}^2 + C_{14}^2 - B_{34}C_{13}=0$. A long but
direct calculation shows that the corresponding Lie algebra is
solvable only for $C_{13}\not=0$ and $C_{14}=0$, and in this case
the equations \eqref{differentials} reduce to~\eqref{Azero-ecus-5}
with $r=C_{13}\in \mathbb{R}^*$.

\medskip

\noindent{\bf Case 2.2:} If $B_{13}^2 = B_{12} C_{13}$ and
$B_{12}\not=0$, then $C_{13}=B_{13}^2/B_{12}$ and
$$
\begin{array}{rl}
& c_{125}^4=0 \quad\quad\Leftrightarrow \quad\quad
(3B_{12}^2+2B_{13}^2+B_{12}B_{34})(B_{13}B_{14} - B_{12}C_{14})=0,\\[7pt]
& c_{145}^2=0 \quad\quad\Leftrightarrow \quad\quad
(B_{13}^2-B_{12}B_{34})(B_{13}B_{14}- B_{12}C_{14})=0,
\end{array}
$$
which implies that
$$
(B_{12}^2+B_{13}^2)(B_{13}B_{14} - B_{12}C_{14})=0.
$$
Since $B_{12}\not=0$ we get $C_{14}=B_{13}B_{14}/B_{12}$. A direct
calculation shows that the Jacobi identity holds if and only if
$$
(B_{12} + B_{34})(B_{13}^2 + B_{14}^2 - B_{12} B_{34})=0.
$$
We distinguish two cases:

\medskip

\noindent{\bf (i)} $B_{34}=-B_{12}\not=0$: in this case the
equations \eqref{differentials} reduce to~\eqref{Azero-ecus-3}. In
fact, take $\theta\in(0,\frac{2\pi}{3})$ such that
$\cos3\theta=B_{14}(B_{12}^2+B_{13}^2+B_{14}^2)^{-\frac12}$,
$\sin3\theta=(B_{12}^2+B_{13}^2)^{\frac12}(B_{12}^2+B_{13}^2+B_{14}^2)^{-\frac12}$.
Then, from \eqref{differentials} we have that the new basis
\[
\begin{array}{rcl}
f^1 \!\!\!&=&\!\!\! \cos\theta\, e^1 + \sin\theta\,
B_{13}(B_{12}^2+B_{13}^2)^{-\frac12} e^2 -
\sin\theta\, B_{12} (B_{12}^2+B_{13}^2)^{-\frac12} e^3,\\[9pt]
f^2 \!\!\!&=&\!\!\! -\sin\theta\,
B_{13}(B_{12}^2+B_{13}^2)^{-\frac12} e^1 +
\cos\theta\, e^2 - \sin\theta\, B_{12}(B_{12}^2+B_{13}^2)^{-\frac12} e^4,\\[9pt]
f^3 \!\!\!&=&\!\!\! \sin\theta\, B_{12}
(B_{12}^2+B_{13}^2)^{-\frac12} e^1 +
\cos\theta\, e^3 - \sin\theta\, B_{13} (B_{12}^2+B_{13}^2)^{-\frac12} e^4,\\[9pt]
f^4 \!\!\!&=&\!\!\! \sin\theta\, B_{12}
(B_{12}^2+B_{13}^2)^{-\frac12} e^2 +
\sin\theta\, B_{13} (B_{12}^2+B_{13}^2)^{-\frac12} e^3 + \cos\theta\, e^4,\\[9pt]
f^5 \!\!\!&=&\!\!\! e^5,
\end{array}
\]
satisfies~\eqref{Azero-ecus-3} for
$a=\displaystyle\epsilon\frac{B_{13}}{B_{12}}\sqrt{B_{12}^2+B_{13}^2+B_{14}^2}$
and $r=\epsilon\sqrt{B_{12}^2+B_{13}^2+B_{14}^2}$ with $\epsilon=\pm
1$. Moreover, $f^{12}+f^{34}=e^{12}+e^{34}$,
$f^{13}+f^{42}=e^{13}+e^{42}$ and $f^{14}+f^{23}=e^{14}+e^{23}$.

\medskip
\noindent{\bf (ii)} If $B_{34}\not= -B_{12}$ then $B_{34}=(B_{13}^2
+ B_{14}^2)/ B_{12}$. Now, a long but direct calculation shows that
the corresponding Lie algebra is solvable only for $B_{14}=0$, and
in this case the equations \eqref{differentials} reduce to
\eqref{Azero-ecus-7} with $r=B_{12}\not=0$ and $a=B_{13}\in
\mathbb{R}$. This completes the proof of the case $A=0$.

\medskip

Suppose next that $A\not=0$. By Corollary~\ref{red-equations-2} the
coefficients $B_{15}, B_{25}, B_{35}, C_{15}, C_{25}$ and $D_{15}$
are determined by the coefficients $A, B_{12}, B_{13}, B_{14},
B_{34}, C_{13}$ and $C_{14}$. A direct calculation shows that
$$c_{234}^2=2 A\, B_{14},\quad\quad c_{234}^3=2A\, C_{14}, \quad\quad   c_{234}^4=
-A(B_{12}-B_{34}+2C_{13}),
$$
so the Jacobi identity implies that $B_{14}=C_{14}=0$ and
$B_{12}=B_{34}-2C_{13}$. Therefore, there are only the four
remaining coefficients $A, B_{13}, B_{34}$ and $C_{13}$. Moreover,
from \eqref{B15-B25-B35} and \eqref{C15-C25-D15} we get
$B_{15}=C_{15}=D_{15}=0$, thus the differentials of $e^i$ have the
form
\begin{equation}\label{differentials-Anonzero}
\left\{
\begin{array}{l}
d e^1= A e^{14} +  A e^{23},\\[7pt]
d e^2= (B_{34}-2C_{13}) e^{12} + B_{13} e^{13} + (2 A + B_{13})
e^{24} + B_{25} e^{25} + B_{34} e^{34}
+ B_{35} e^{35},\\[7pt]
d e^3= (3 A + B_{13}) e^{12} + C_{13} e^{13} - (2B_{34}-3C_{13})
e^{24}+
C_{25} e^{25} - (A + B_{13}) e^{34} - B_{25} e^{35},\\[7pt]
d e^4=  (B_{34} - C_{13})\, ( e^{14} +
e^{23}),\\[7pt]
d e^5 = -2 e^{14} -2  e^{23}.
\end{array}
\right.
\end{equation}

It can be proved directly that the Jacobi identity for
\eqref{differentials-Anonzero} is satisfied if and only if
\begin{equation}\label{sist-1}
(3 A + B_{13}) B_{35} - B_{13} C_{25}=0, \quad
 (-2 B_{34} + 3 C_{13}) B_{35} - B_{34} C_{25}=0,
\end{equation}
and
\begin{equation}\label{sist}
\begin{array}{rl}
& B_{25} - 2 A\, B_{34} - 2 B_{13} B_{34} + 2 B_{13}
C_{13}=0,\\[7pt]
& 6 A^2 + 5 A\, B_{13} + B_{13}^2 + 2 B_{34}^2 - 7 B_{34} C_{13} +6 C_{13}^2 +C_{25}=0,\\[7pt]
& -6 A\, B_{25} - 2 B_{13} B_{25} + B_{34} C_{25} - 3 C_{13} C_{25}=0,\\[7pt]
& A\, B_{13} + B_{13}^2 - B_{35} + B_{34} C_{13}=0,\\[7pt]
& 2 B_{13} B_{25} - B_{34} B_{35} + 3 B_{35} C_{13}=0,\\[7pt]
& 4 B_{25} B_{34} - 6 B_{25} C_{13} + 3 A\, C_{25} + 2 B_{13} C_{25}=0,\\[7pt]
& 2 B_{25} B_{34} - 3 A\, B_{35} - 2 B_{13} B_{35}=0.
\end{array}
\end{equation}

Let $\rho=-3(  A B_{34} + B_{34} B_{13} - B_{13} C_{13})$ be the
determinant of the linear system~\eqref{sist-1}. If $\rho\not=0$
then $B_{35}=C_{25}=0$, and the first equation in \eqref{sist}
implies that $B_{25}=-{2\over 3}\rho\not= 0$, so in order to be
satisfied the remaining equations in \eqref{sist} we must have
$B_{13}=B_{34}=C_{13}=0$, which is in contradiction with the second
equation in \eqref{sist}. Therefore, $\rho=0$ and there is $a\in
\mathbb{R}$ such that $-2 B_{34} + 3 C_{13} =3 a\, A  + a\, B_{13} $
and $B_{34}=a\, B_{13}$. From the first equation in \eqref{sist} we
get
\begin{equation}\label{sist-2}
B_{34}=a\, B_{13},\quad\quad   C_{13} =a(A + B_{13}),\quad\quad
B_{25}=0,
\end{equation}
and \eqref{sist-1}-\eqref{sist} reduce to
\begin{equation}\label{sist-red}
\begin{array}{rl}
& 3 A\, B_{35} + B_{13} B_{35} - B_{13} C_{25}=0, \quad
 C_{25} + (6 A^2 + 5 A\, B_{13} + B_{13}^2) (1+a^2)=0,\\[7pt]
& B_{35} - B_{13}(A + B_{13}) (1+a^2) =0,\quad (3 A + 2 B_{13})
C_{25}=0, \quad (3 A + 2 B_{13}) B_{35}=0.
\end{array}
\end{equation}
Notice that $B_{13} \not= -{3\over 2}A$ implies $B_{35}=C_{25}=0$,
and from the third equation in \eqref{sist-red} it follows that
$B_{13}=0,-A$, which does not solve the second equation in
\eqref{sist-red}. Therefore, $B_{13}= -{3\over 2}A$ and the solution
to~\eqref{sist-red} is $B_{35} = {3\over 4} A^2 (1+a^2)$ and $C_{25}
= -B_{35}$. From \eqref{differentials-Anonzero} and \eqref{sist-2}
we get that the new basis $f^1=\nu(a\, e^1+e^4)$, $f^2=\nu(a\,
e^2-e^3)$, $f^3=\nu(e^2+a\, e^3)$, $f^4=\nu(-e^1+a\, e^4)$,
$f^5=e^5$, where $\nu=-(1+a^2)^{-1/2}$, satisfies
\eqref{Azero-ecus-2} with $r=-\nu^{-1} A/2$ and
$f^{12}+f^{34}=e^{12}+e^{34}$, $f^{13}+f^{42}=e^{13}+e^{42}$,
$f^{14}+f^{23}=e^{14}+e^{23}$. That is to say, the case $A\not=0$
reduces to \eqref{Azero-ecus-2} and the proof of the proposition is
complete.
\end{proof}

\begin{remark}\label{clasif-more-gen}
Notice that the condition $[\frg,\frg]\not= \frg$ implies that
$\frg^*$ has a nonzero element which is closed, so the proof of
Proposition~\ref{red-equations} still holds for Lie algebras
satisfying $[\frg,\frg]\not= \frg$, even when they are not solvable.
Moreover, the proof of Proposition~\ref{8families} shows that if
such a Lie algebra admits hypo-contact structure then it belongs to
Case 2.1 with $B_{34}\not= 0 = B_{14}= C_{13}=C_{14}$, Case 2.1 with
$B_{34} C_{13} C_{14}\not= 0 = B_{14}$, or Case~2.2 with
$(B_{34}+B_{12})B_{14}\not= 0$.
\end{remark}

\medskip

Now, using Proposition~\ref{8families}, we obtain the classification
of solvable hypo-contact Lie algebras.

\medskip

\noindent{\bf Proof of Theorem~\ref{clasification}~:} A solvable Lie
algebra with a hypo-contact structure belongs, by
Proposition~\ref{8families}, to one of the six families
\eqref{Azero-ecus-1}--\eqref{Azero-ecus-7}. Therefore, in order to
prove the theorem, it suffices to show that $\frh_1,\ldots,\frh_5$
are the Lie algebras underlying these families. For the family
\eqref{Azero-ecus-1}, the new basis
$$
\alpha^1= 2e^4-3re^5, \quad \alpha^2=5e^3, \quad
\alpha^3=2re^2,\quad  \alpha^4= -3e^4-3re^5, \quad \alpha^5=re^1
$$
satisfies
$$
d \alpha^1= -2\alpha^{15}-\alpha^{23},\quad d \alpha^2=
-\alpha^{25},\quad d \alpha^3=  -\alpha^{35},\quad d \alpha^4=
3\alpha^{45},\quad d \alpha^5 = 0.
$$
Therefore, any Lie algebra in the family \eqref{Azero-ecus-1} is
isomorphic to $\frh_2$.

Any Lie algebra in the family \eqref{Azero-ecus-2} is isomorphic to
$\frh_3$. In fact, with respect to the new basis
$$
\alpha^1= re^4, \quad \alpha^2=\sqrt{2}re^3, \quad
\alpha^3=\sqrt{2}re^2,\quad \alpha^4= re^1, \quad
\alpha^5=3re^4+3r^2e^5,
$$
the equations \eqref{Azero-ecus-2} become
$$
d \alpha^1= -2\alpha^{14}-\alpha^{23},\quad d \alpha^2=
-\alpha^{24}-\alpha^{35},\quad d \alpha^3=
\alpha^{25}-\alpha^{34},\quad d \alpha^4= d \alpha^5 = 0.
$$

Any Lie algebra in the family \eqref{Azero-ecus-3} is isomorphic to
$\frh_4$, because with respect to the new basis
\[
\begin{array}{rcl}
\alpha^1&=& 2 e^2+ r\, e^5,\\[5pt]
\alpha^2&=&\displaystyle\frac{\sqrt{3}\,a}{r}e^1-\frac{\sqrt{a^2+r^2}}{r}e^2+
\sqrt{3}\, e^4+\sqrt{a^2+r^2}\, e^5,\\[7pt]
\alpha^3&=&\displaystyle-\frac{\sqrt{3}\,
a}{r}e^1-\frac{\sqrt{a^2+r^2}}{r}e^2
-\sqrt{3}\, e^4+\sqrt{a^2+r^2}\,e^5,\\[7pt]
\alpha^4&=& -2a\, e^2+ 2r\, e^3,\\[7pt]
\alpha^5&=&-\sqrt{3}\sqrt{a^2+r^2}\, e^1+ a\, e^2- r\, e^3,
\end{array}
\]
the equations \eqref{Azero-ecus-3} become
$$
d\alpha^1=-\alpha^{14},\quad d\alpha^2=-\alpha^{25},\quad
d\alpha^3=\alpha^{34}+\alpha^{35},\quad d\alpha^4=d\alpha^5=0.
$$

It is clear that $\frh_1$ is obtained when $a=b=0$ in the family
\eqref{Azero-ecus-4}. If $(a,b)\not= (0,0)$ then, after the change
of basis $f^1=e^2$, $f^2=e^1$, $f^3=e^4$, $f^4=e^3$, $f^5=e^5$ if
necessary, we can suppose that $b\not=0$.

Now, let us fix a pair $(a,b)$ with $b\not=0$. Let us consider
equations \eqref{Azero-ecus-3} for the pair $(a,r=b\not=0)$ in terms
of $e^1,\ldots,e^5$. Then, the new basis given by
\begin{equation} \label{24-25-bis}
\begin{array}{rcl}
f^1&=& \sin\sigma\, e^1 - \cos\sigma(\cos\theta\,e^2 -
\sin\theta\,e^3),\\[4pt]
f^2&=& \cos\sigma\, e^1 + \sin\sigma(\cos\theta\,e^2 -
\sin\theta\,e^3),\\[4pt]
f^3&=& \sin\sigma(\sin\theta\,e^2 + \cos\theta\,e^3) + \cos\sigma\,
e^4,\\[4pt]
f^4&=& -\cos\sigma(\sin\theta\,e^2 + \cos\theta\,e^3) + \sin\sigma\,
e^4,\\[4pt]
f^5&=&e^5,
\end{array}
\end{equation}
where $\theta\in (0,2\pi)$ is such that
$\cos\theta=a/\sqrt{a^2+b^2}$, $\sin\theta=b/\sqrt{a^2+b^2}$ and
$\sigma=(\theta-\pi)/3$, satisfies equations of the form
\eqref{Azero-ecus-4} for the given pair $(a,b)$. Therefore, the Lie
algebras underlying \eqref{Azero-ecus-4} are all isomorphic to
$\frh_4$.

Any Lie algebra in the family \eqref{Azero-ecus-5} is isomorphic to
$\frh_5$, because with respect to the new basis
$$
\alpha^1= e^4- {r\over 2}e^5, \quad \alpha^2=-\sqrt{2}\, e^2, \quad
\alpha^3=-e^4-{r\over 2}e^5,\quad \alpha^4= \sqrt{2}\,r\,e^3, \quad
\alpha^5=r\,e^1,
$$
the equations \eqref{Azero-ecus-5} transform into
$$
d \alpha^1= -\alpha^{15}-\alpha^{24},\quad d \alpha^2=
-\alpha^{34},\quad d \alpha^3=  \alpha^{35},\quad d \alpha^4=
-\alpha^{45},\quad d \alpha^5 = 0.
$$

Also, the Lie algebras underlying \eqref{Azero-ecus-7} are all
isomorphic to $\frh_5$.  In fact, let us fix a pair $(a,r)$ with
$r\not=0$, and let us consider equations \eqref{Azero-ecus-5} for
$s=(a^2+r^2)/r$ in terms of $e^1,\ldots,e^5$. Then, the new basis
given by
\begin{equation} \label{26-27}
f^1= e^1, \quad\
 f^2= \cos\theta\, e^2 + \sin\theta\, e^3, \quad \
f^3= -\sin\theta\, e^2 + \cos\theta\, e^3,  \quad \ f^4= e^4, \quad
\ f^5=e^5.
\end{equation}
where $\theta\in (0,2\pi)$ is such that
$\cos\theta=a/\sqrt{a^2+r^2}$ and $\sin\theta=r/\sqrt{a^2+r^2}$,
satisfies equations of the \eqref{Azero-ecus-7} for the given pair
$(a,r)$. Therefore, the Lie algebras underlying \eqref{Azero-ecus-7}
are all isomorphic to $\frh_5$. \QED\medskip

Diatta obtains in \cite{D} a list of solvable contact Lie algebras
in dimension 5 and many of them have non-trivial center. Notice that
$\frh_1,\ldots,\frh_5$ correspond to the Lie algebras {\bf 1}, {\bf
4}$(p=1,q=-3)$, {\bf 22}, {\bf 18}$(p=q=-1)$ and {\bf 15}$(p=-1)$,
respectively, in Diatta's list and that the center of the solvable
Lie algebras $\frh_2,\ldots,\frh_5$ is trivial.

\begin{definition}\label{hypo-contact-isomorphism}
Let $\frg$ and $\tilde\frg$ be Lie algebras endowed with hypo
structures $(\eta,\omega_i)$ and $(\tilde\eta,\tilde\omega_i)$,
respectively. We say that the hypo structures are \emph{equivalent
by rotation} if there is an isomorphism of Lie algebras $F\colon
\frg\longrightarrow \tilde\frg$ such that $\eta=F^*\tilde\eta$,
$\omega_3=F^*\tilde\omega_3$, $\omega_1=\cos\theta\,
F^*\tilde\omega_1 -\sin\theta\, F^*\tilde\omega_2$ and
$\omega_2=\sin\theta\, F^*\tilde\omega_1 +\cos\theta\,
F^*\tilde\omega_2$, for some $\theta$.
\end{definition}

If two hypo structures are equivalent by rotation via $F$ then $F$
preserves the induced metrics. In the following result we show which
families of hypo-contact structures given in
Proposition~\ref{8families} are equivalent by rotation.

\begin{proposition}\label{isomorphic-hypo}
Any hypo-contact structure in the family \eqref{Azero-ecus-3}
(respectively, \eqref{Azero-ecus-7}) is equivalent by rotation to a
hypo-contact structure in the family \eqref{Azero-ecus-4} for some
$(a,b)\not=(0,0)$ (respectively, \eqref{Azero-ecus-5}).
\end{proposition}

\begin{proof}
Let us consider equations \eqref{Azero-ecus-3} for $(a,r\not=0)$ in
terms of $e^1,\ldots,e^5$, and equations \eqref{Azero-ecus-4} for
$(a,b=r\not=0)$ in terms of $f^1,\ldots,f^5$. Then,
\eqref{24-25-bis} defines an isomorphism of Lie algebras such that
$e^5=f^5$, $e^{14}+e^{23}=f^{14}+f^{23}$ and
$$
e^{12}+e^{34}=\cos(-\theta)(f^{12}+f^{34})
-\sin(-\theta)(f^{13}+f^{42}),\qquad e^{13}+e^{42}=\sin(-\theta)
(f^{12}+f^{34}) +\cos(-\theta) (f^{13}+f^{42}).
$$
This shows that any hypo-contact structure in family
\eqref{Azero-ecus-3} is equivalent by rotation to the hypo-contact
structure in the family \eqref{Azero-ecus-4} for $(a,b=r\not=0)$.

Let us consider equations \eqref{Azero-ecus-7} for $(a,r\not=0)$ in
terms of $f^1,\ldots,f^5$, and equations \eqref{Azero-ecus-5} for
$s=(a^2+r^2)/r$ in terms of $e^1,\ldots,e^5$. Then, \eqref{26-27}
defines an isomorphism of Lie algebras such that $f^5=e^5$,
$f^{14}+f^{23}=e^{14}+e^{23}$ and
$$
f^{12}+f^{34}=\cos\theta(e^{12}+e^{34})
-\sin\theta(e^{13}+e^{42}),\quad \quad \
f^{13}+f^{42}=\sin\theta(e^{12}+e^{34}) +\cos\theta(e^{13}+e^{42}).
$$
This shows that any hypo-contact structure in family
\eqref{Azero-ecus-7} is equivalent by rotation to a hypo-contact
structure in the family \eqref{Azero-ecus-5}.
\end{proof}

\begin{remark}\label{no-iso}
A direct calculation shows that any equivalence by rotation between
hypo-contact structures in the families \eqref{Azero-ecus-3} and
\eqref{Azero-ecus-4} must have $\theta\not=0$. The same holds for
any equivalence by rotation between hypo-contact structures in the
families \eqref{Azero-ecus-5} and \eqref{Azero-ecus-7}.
\end{remark}

\section{$K$-contact and $\eta$-Einstein structures} \label{Sasak-eta-Einstein}

The Lie algebras described in Proposition \ref{8families} cannot be
Einstein \cite{D}. In this section, we show that any Lie algebra of
\eqref{Azero-ecus-2} and the Lie algebra of \eqref{Azero-ecus-4} for
$a=b=0$ with the hypo-contact structure defined by \eqref{structure}
are the only ones which are $\eta$-Einstein \cite{Ok} or,
equivalently \cite{ConS}, Sasakian. Moreover, we prove that these
Lie algebras are also $K$-contact \cite{Blair}.

Consider an odd-dimensional manifold $M$ with a contact form $\eta$
and associated metric $g$. Denote by $\xi$ the vector field on $M$
dual of $\eta$. Recall that $(M,g,\eta)$ is said to be
\emph{$K$-contact} if $\xi$ is a Killing vector field; $(M,g,\eta)$
is called \emph{$\eta$-Einstein} \cite{Ok} if there exist smooth
functions $\tau,\nu$ on $M$ such that the Ricci curvature tensor
satisfies
$$
{\mbox {Ric}} (X, Y) = \tau\, g (X, Y) + \nu\, \eta (X) \eta (Y),
$$
for any vector fields $X, Y$ on $M$. The functions $\tau$ and $\nu$
are uniquely determined by
$$
s = \tau\, \dim M + \nu, \quad {\mbox Ric} (\xi, \xi) = \tau + \nu,
$$
where $s$ denotes the scalar curvature of $g$. When $\nu=0$ we have
the well-known Einstein condition. In our situation, $\tau$ and
$\nu$ are constant and $\xi$ is the vector dual to $\eta=e^5$.

In the following proposition we distinguish the solvable Lie
algebras of Proposition \ref{8families} for which the hypo-contact
structure defined by \eqref{structure} is $K$-contact.

\begin{proposition}\label{k-contact}
Let $\frg$ be a solvable Lie algebra with a $K$-contact hypo
structure $(\eta, \omega_1,\omega_2,\omega_3)$ and with a basis
$e^1,\ldots,e^5$ for ${\frak g}^*$ which satisfies the conditions of
Proposition \ref{red-equations}. Then, its structure equations
reduce to \eqref{Azero-ecus-2} or \eqref{Azero-ecus-4} for $a=b=0$.
Moreover, they are Sasakian and $\eta$-Einstein.
\end{proposition}

\begin{proof}
Denote by $e_1,\ldots,e_5$ the dual basis of $e^1,\ldots,e^5$. The
vector $e_5$ is a Killing vector if and only if $g (\nabla_Y e_5, Z)
+ g(\nabla_Z e_5, Y) =0$ for any $Y, Z \in {\mathfrak g}$ or,
equivalently, $g ([e_5, e_i], e_j) + g([e_5, e_j], e_i) =0$, for any
$i,j$. Since the basis $e_1,\ldots,e_5$ is orthonormal, the latter
condition is equivalent to
\begin{equation}\label{killing}
de^j(e_i, e_5) + de^i(e_j,e_5) =0,
\end{equation}
for any $i,j$. In particular, for $(i,j)=(1,4)$, since $e^1$ is
closed, from Proposition~\ref{8families} we get that $de^4(e_1,
e_5)\not=0$ for the families \eqref{Azero-ecus-1},
\eqref{Azero-ecus-3}, \eqref{Azero-ecus-4} unless $a=b=0$,
\eqref{Azero-ecus-5} and \eqref{Azero-ecus-7}. Therefore, these
families are not $K$-contact.

On the other hand, it is clear that \eqref{Azero-ecus-2} and
\eqref{Azero-ecus-4} for $a=b=0$ satisfy \eqref{killing}. Moreover,
the equations
$$
d\eta=-2\omega_3,\quad d \omega_1= \lambda\,
\omega_2\wedge\eta,\quad d\omega_2=-\lambda\, \omega_1\wedge\eta,
$$
are satisfied for these families, where $\lambda = - 3 r^2$ for the
family \eqref{Azero-ecus-2}, and $\lambda=0$ for
\eqref{Azero-ecus-4} with $a=b=0$. Therefore, $(\omega_i, g)$ is
$\eta$-Einstein (see \cite{BGM}). But Theorem $14$ of \cite{ConS}
asserts that a hypo structure is Sasakian if and only if it is
$\eta$-Einstein, which completes the proof.
\end{proof}

\medskip

Furthermore, concerning the $\eta$-Einstein property we have:

\begin{proposition}\label{families-eta-Einstein}
The families \eqref{Azero-ecus-2} and \eqref{Azero-ecus-4} for
$a=b=0$ with the hypo-contact structure defined by~\eqref{structure}
are the only ones which are $\eta$-Einstein.
\end{proposition}

\begin{proof}
First, notice that the $\eta$-Einstein condition is preserved under
equivalence by rotation. Therefore, from
Proposition~\ref{isomorphic-hypo} it suffices to study the families
\eqref{Azero-ecus-1}, \eqref{Azero-ecus-2}, \eqref{Azero-ecus-4} and
\eqref{Azero-ecus-5}. By direct computation one can check that the
nonzero components of the Ricci tensor for these four families are
given respectively by
$$
\begin{array}{l}
{\mbox {Ric}} (e_i, e_i) = \left\{
\begin{array}{l}
-\frac12\bigl(9r^4+18r^2+4\bigr), \quad i = 1,\\[5pt]
- (3 r^2+2), \quad  i = 2,3,\\[5pt]
\frac 12\bigl(9r^4+6r^2-4\bigr), \quad i=4,\\[5pt]
-\frac12\bigl(9r^4-8\bigr),\quad  i = 5.
\end{array} \right.\\[40pt]

{\mbox {Ric}} (e_i, e_i) = \left\{
\begin{array}{l}
-2(3 r^2+1), \quad i = 1, \ldots, 4,\\[5pt]
4, \quad  i = 5.
\end{array} \right.\\[20pt]

{\mbox {Ric}} (e_i, e_i ) = \left\{
\begin{array}{l}
-\frac 12\bigl((a^2+b^2)^2+8(a^2+b^2)+4\bigr),\quad i=1,2,\\[5pt]
\frac 12\left((a^2+b^2)^2-4\right),\quad i=3,4,\\[5pt]
4-(a^2+b^2)^2,\quad i=5.
 \end{array} \right.

\end{array}
$$
$$
{\mbox {Ric}} (e_i, e_i) = \left\{
\begin{array}{l}
- \frac 18 \bigl(r^4+16 r^2+16\bigr), \quad i = 1, 3,\\[5pt]
\frac 18\left(r^4-16\right), \quad i = 2, 4,\\[5pt]
-\frac 14\left(  r^4-16\right), \quad i = 5.
\end{array} \right.
$$
Therefore, $(\eta=e^5, g)$ is $\eta$-Einstein only for the families
\eqref{Azero-ecus-2} and \eqref{Azero-ecus-4} for $a=b=0$. Notice
that
$$
{\mbox {Ric}}(X,Y)= -2(1+3r^2) g (X, Y) + 6(1+r^2) \eta(X) \eta(Y)
$$
for the family \eqref{Azero-ecus-2}, and
$$
{\mbox {Ric}} (X, Y) = -2 g (X, Y) + 6 \eta (X) \eta (Y)
$$
for \eqref{Azero-ecus-4} with $a=b=0$.
\end{proof}

\begin{remark}
We must notice that for the families \eqref{Azero-ecus-2} and
\eqref{Azero-ecus-4} with $a=b=0$, $\eta$-Einstein condition implies
$K$-contact property, which in general  is not true.
\end{remark}

As a consequence of the above propositions and
Theorem~\ref{clasification} we conclude

\begin{corollary}\label{clasif-Kcontact-eta-Einstein}
The only 5-dimensional solvable Lie algebras admitting a $K$-contact
hypo structure or a hypo-contact $\eta$-Einstein structure are
$\frh_1$ and $\frh_3$.
\end{corollary}

Hypo-contact structures are related to the contact Calabi-Yau
structures introduced recently in~\cite{TV}. A {\it contact
Calabi-Yau structure} on a ($2n+1$)-dimensional manifold $M$ is a
triple $(\eta,J,\epsilon)$, where $(\eta,J)$ is a Sasakian structure
and $\epsilon\in \Lambda^{n,0}_J(\ker \eta)$ is a nowhere vanishing
basic form on $\ker \eta$ such that $d\epsilon=0$ and
$\epsilon\wedge \bar{\epsilon}=(-1)^{n(n+1)\over 2} i^n (d\eta)^n$.
In dimension 5, if $(\eta,J,\epsilon)$ is a contact Calabi-Yau
structure then the quadruplet $(-\eta,\, \omega_1=\Re \epsilon,\,
\omega_2=\Im \epsilon,\, \omega_3=-{1\over 2} d\eta)$ defines a
hypo-contact structure on $M$ for which $d\omega_1=d\omega_2=0$ and
the metric induced by $(\eta,J)$ is $\eta$-Einstein with $\tau=-2$
and $\nu=6$~\cite[Corollary 3.7]{TV}..

Tomassini and Vezzoni classify 5-nilmanifolds admitting invariant
contact Calabi-Yau structure and prove that, up to isomorphism, the
only (non-trivial) 5-dimensional nilpotent Lie algebra admitting
hypo-contact structure is  $\frh_1$. This result  also follows
directly from the fact that $\frh_1$ is the only 5-dimensional
nilpotent Lie algebra admitting a  Sasakian
structure~\cite[Corollary 5.5]{U}. Next we show that there are no
5-dimensional solvable non-nilpotent Lie algebras admitting contact
Calabi-Yau structure.

\begin{proposition}\label{clasif-contact-Calabi-Yau}
Let $\frg$ be a 5-dimensional Lie algebra such that
$[\frg,\frg]\not= \frg$ admitting a contact Calabi-Yau structure.
Then, $\frg$ is isomorphic to  the nilpotent Lie algebra $\frh_1$.
\end{proposition}

\begin{proof}
It is sufficient to prove that if $\frg$ admits a hypo-contact
structure $(\eta, \omega_i)$ with $\omega_1$ and $\omega_2$ closed,
then $\frg$ is isomorphic to $\frh_1$. If $\frg$ is solvable then
from Proposition~\ref{8families} one can see directly that
$\omega_1$ and $\omega_2$ are both closed only if $a=b=0$ in the
family~\eqref{Azero-ecus-4}, which corresponds to $\frh_1$. Finally,
when $\frg$ is not solvable we apply Remark~\ref{clasif-more-gen}
and a direct calculation shows that $\omega_1$ and $\omega_2$ are
also nonclosed in this case.
\end{proof}

\section{Metrics with holonomy $SU(3)$} \label{hypo-vol-eq}

The purpose of this section is to prove
Theorem~\ref{Calabi-Yau-metrics}, that is, any left-invariant
hypo-contact structure $(\eta, \omega_1, \omega_2, \omega_3)$ on a
solvable Lie group $N$  gives rise to a metric with holonomy $SU(3)$
via the Conti-Salamon evolution equations \eqref{vol}. From now on,
to avoid confusion, we denote the exterior differential on  $N$ by
$\hat d$, and the exterior differential on $N\times I$ by $d$. Then,
the (hypo) evolution equations \eqref{vol} are written as follows
\begin{equation} \label{vol-1}
\left\{ \begin{array} {l}
\partial_t \omega_3(t) =  - \hat d \eta(t),\\
\partial_t ( \omega_2(t) \wedge \eta(t)) =  \hat d \omega_1(t),\\
\partial_t ( \omega_1(t) \wedge \eta(t)) = -  \hat d \omega_2(t).\\
\end{array} \right.
\end{equation}

In order to prove Theorem~\ref{Calabi-Yau-metrics} we first observe
the following fact. Let $(\eta,\omega_i)$ and
$(\tilde\eta,\tilde\omega_i)$ be two hypo-contact structures on a
Lie algebra $\frg$ which are equivalent by rotation in the sense of
Definition~\ref{hypo-contact-isomorphism}. If
$(\tilde\eta(t),\tilde\omega_i(t))$ is a solution of the evolution
equations~\eqref{vol-1} for $(\tilde\eta,\tilde\omega_i)$, then
$\eta(t)=F^*(\tilde\eta(t))$, $\omega_3(t)=F^*(\tilde\omega_3(t))$,
$\omega_1(t)=\cos\theta\, F^*(\tilde\omega_1(t)) -\sin\theta\,
F^*(\tilde\omega_2(t))$ and $\omega_2=\sin\theta\,
F^*(\tilde\omega_1(t)) +\cos\theta\, F^*(\tilde\omega_2(t))$ is a
solution of~\eqref{vol-1} for the hypo-contact structure
$(\eta,\omega_i)$. Therefore, it suffices to prove the theorem up to
equivalence by rotation of the hypo-contact structure.

\bigskip
\noindent{\bf Proof of Theorem~\ref{Calabi-Yau-metrics}~:} From the
observation above, Theorem~\ref{clasification} and
Propositions~\ref{8families} and~\ref{isomorphic-hypo} we shall
concentrate on the families \eqref{Azero-ecus-1},
\eqref{Azero-ecus-2}, \eqref{Azero-ecus-4} and \eqref{Azero-ecus-5},
showing for each case the existence of a solution of the evolution
equations for which the metric associated to the corresponding
integrable $SU(3)$-structure has holonomy group equal to $SU(3)$.

We consider first the $\eta$-Einstein case in detail. In this case
we have that $\hat d \omega_1 = \lambda  \, \omega_2 \wedge e^5$ and
$\hat d \omega_2 =  - \lambda\,  \omega_1 \wedge e^5$, where
$\lambda = - 3 r^2$ for the family~\eqref{Azero-ecus-2} and $\lambda
=0$ for the nilpotent Lie algebra corresponding
to~\eqref{Azero-ecus-4} for $a=b=0$. A solution of the hypo
evolution equations is given by
$$
\eta(t) = \frac{1} {2} f'(t) \, e^5,\qquad \omega_1 (t) = f(t) \,
(e^{12}+e^{34}), \qquad\omega_2 (t) = f(t) \, (e^{13}-e^{24}),
\qquad\omega_3 (t) = f(t) \, (e^{14}+e^{23}),
$$
where $f(t)$ is a function such that $f(0) = 1, f'(0) = 2$ and
satisfies the ordinary differential equation
$$
f f'' + (f')^2 - 2 \lambda f =0.
$$
For $\lambda =0$ one has the explicit solution $f(t)=(1+4t)^{1/2}$
and the Riemannian metric with $SU(3)$-holonomy that one gets is the
one obtained in \cite{GLPS}, namely
 $$
g = (1+4t)^{1/2} \bigl( (e^1)^2 + (e^2)^2 + (e^3)^2 + (e^4)^2\bigr)
+ \frac{1}{1+4t}  (e^5)^2 + dt^2.
 $$

If $\lambda = - 3 r^2 <0$ then, after performing a first
integration, one obtains the first order differential equation
$$
f'(t)=\frac{2}{f(t)} \left(1+r^2-r^2f^3(t)\right)^{1/4}
$$
with initial condition $f(0) = 1$. Therefore, there exists a unique
solution $f(t)$ defined on some open neighbourhood $I$ around $t=0$.
The basis of 1-forms on the manifold $H_3\times I$ given by
$$
\eta^1=\sqrt{f(t)}\, e^1,\quad \eta^2=\sqrt{f(t)}\, e^2,\quad
\eta^3=\sqrt{f(t)}\, e^3,\quad \eta^4=\sqrt{f(t)}\, e^4,\quad
\eta^5=\frac{f'(t)}{2}\, e^5,\quad \eta^6=dt,
$$
is orthonormal with respect to the Riemnannian metric associated to
the corresponding integrable $SU(3)$-structure on $H_3\times I$. By
computing the curvature forms $\Omega^i_j$ and applying the
Ambrose-Singer theorem, one can see that the holonomy group is
actually $SU(3)$. In fact, a direct calculation shows that,  for
each $r$,  the curvature forms
$$
\begin{array} {rcl}
\Omega^1_2 \!\!&\!\!\!=\!\!\!&\!\!  -\Omega^3_4 = -\frac{4r^2f(t)+(f'(t))^2}{4f^2(t)} (\eta^{12} - \eta^{34}),\\[7pt]
\Omega^1_3 \!\!&\!\!\!=\!\!\!&\!\!  \Omega^2_4 = -\frac{4r^2f(t)+(f'(t))^2}{4f^2(t)} (\eta^{13} + \eta^{24}),\\[7pt]
\Omega^1_4 \!\!&\!\!\!=\!\!\!&\!\!
-\frac{4r^2f(t)+(f'(t))^2}{2f^2(t)} (2\eta^{14} + \eta^{23})
+ \frac{(f'(t))^2-2f(t)f''(t)}{2f^2(t)} \, \eta^{56},\\[7pt]
\Omega^1_5 \!\!&\!\!\!=\!\!\!&\!\! \Omega^4_6 = \frac{(f'(t))^2-2f(t)f''(t)}{4f^2(t)} (\eta^{15} + \eta^{46}),\\[7pt]
\Omega^1_6 \!\!&\!\!\!=\!\!\!&\!\! -\Omega^4_5 =
\frac{(f'(t))^2-2f(t)f''(t)}{4f^2(t)} (\eta^{16} - \eta^{45}),\\[7pt]
\Omega^2_3 \!\!&\!\!\!=\!\!\!&\!\!
-\frac{4r^2f(t)+(f'(t))^2}{2f^2(t)} ( \eta^{14} + 2 \eta^{23}) +
\frac{(f'(t))^2-2f(t)f''(t)}{2f^2(t)}\,
\eta^{56},\\[7pt]
\Omega^2_5 \!\!&\!\!\!=\!\!\!&\!\! \Omega^3_6 = \frac{(f'(t))^2-2f(t)f''(t)}{4f^2(t)} (\eta^{25} + \eta^{36}),\\[7pt]
\Omega^2_6 \!\!&\!\!\!=\!\!\!&\!\! -\Omega^3_5 =
\frac{(f'(t))^2-2f(t)f''(t)}{4f^2(t)} (\eta^{26} -
\eta^{35}),\\[7pt]
\end{array}
$$
are linearly independent at $t=0$, since  $f(0)=1$, $f'(0)=2$,
$f''(0)=-2(3r^2+2)$ and $f'''(0)=24(r^2+1)$.  Therefore, any
$\eta$-Einstein hypo-contact structure gives rise to a metric whose
holonomy group is equal to $SU(3)$.

For the family \eqref{Azero-ecus-4} a solution for the hypo
evolution equations is given by
$$
\begin{array} {lll}
\eta(t) \!\!&\!\!\!= \frac 12 f'(t) e^5, \quad\quad\quad &
\omega_1 (t) = \frac 12 f^2(t) f'(t) \,  e^{12} + \frac 2{f'(t)} \,  e^{34},\\[6pt]
\omega_2(t) \!\!&\!\!\!=   f(t) \, (e^{13} - e^{24}),
\quad\quad\quad & \omega_3 (t) =  f(t) \, (e^{14} + e^{23}),
\end{array}
$$
where $f(t)$ satisfies the differential equation
$$
4\rho +{f'}^3+f f'f'' =0
$$
with initial conditions $f(0) = 1$ and $f'(0)=2$, where
$\rho=a^2+b^2$. After performing a first integration, one obtains
the first order differential equation
$$
f'(t)=\frac{\left(8+ 4\rho - 4\rho f^3(t)\right)^{1/3}}{f(t)}
$$
with initial conditions $f(0) = 1$. Therefore, there exists a unique
solution $f(t)$ defined on some open neighbourhood $I$ around $t=0$.

A similar computation as in the $\eta$-Einstein case above, shows
that the holonomy group of the metric associated to the
corresponding integrable $SU(3)$-structure on $H_4\times I$ is also
equal to $SU(3)$. In fact, the curvature forms $\Omega^1_2,
\Omega^1_3, \Omega^1_4, \Omega^1_5, \Omega^1_6, \Omega^2_3,
\Omega^2_5$ and $\Omega^2_6$ take the following values when $t=0$:
$$
\begin{array} {rcl}
(\Omega^1_2)_{\mid_{t=0}} \!\!&\!\!\!=\!\!\!&\!\!
-\frac{(\rho-2)^2}{4}
(e^{12} - e^{34}),\\[7pt]
(\Omega^1_3)_{\mid_{t=0}} \!\!&\!\!\!=\!\!\!&\!\!
\frac{(\rho-2)^2-8}{4} (e^{13} + e^{24}) +
b \rho (e^{15}+e^{46}) + a \rho ( e^{25} +  e^{36}) ,\\[7pt]
(\Omega^1_4)_{\mid_{t=0}} \!\!&\!\!\!=\!\!\!&\!\!
\frac{(\rho-2)^2-12}{2} e^{14}
-a \rho( e^{15} + e^{46})- 2 e^{23} + b \rho (e^{25} + e^{36})- \frac{(\rho+2)(\rho-6)}{2} e^{56},\\[7pt]
(\Omega^1_5)_{\mid_{t=0}} \!\!&\!\!\!=\!\!\!&\!\! b\rho(
e^{13}+e^{24}) - a \rho (e^{14} -  e^{23})
- \frac{(\rho+2)(5\rho-6)}{4} (e^{15} + e^{46}), \\[7pt]
(\Omega^1_6)_{\mid_{t=0}} \!\!&\!\!\!=\!\!\!&\!\!
\frac{3(\rho+2)^2}{4} (e^{16} - e^{45}),\\[7pt]
(\Omega^2_3)_{\mid_{t=0}} \!\!&\!\!\!=\!\!\!&\!\! -2 e^{14} + a
\rho(e^{15}+ e^{46}) + \frac{(\rho-2)^2-12}{2} e^{23} -b\rho( e^{25}
+ e^{36})  - \frac{(\rho+2)(\rho-6)}{2}
e^{56},\\[7pt]
(\Omega^2_5)_{\mid_{t=0}} \!\!&\!\!\!=\!\!\!&\!\! a \rho( e^{13}+
e^{24}) + b \rho (e^{14} -  e^{23})
-\frac{(\rho+2)(5\rho-6)}{4} (e^{25} + e^{36}),\\[7pt]
(\Omega^2_6)_{\mid_{t=0}} \!\!&\!\!\!=\!\!\!&\!\!
\frac{3(\rho+2)^2}{4} (e^{26} - e^{35}),
\end{array}
$$
where $e^6$ denotes the 1-form $dt$ evaluated at $t=0$, and they are
linearly independent if and only if $\rho\not= 2,6$. Moreover, if
$\rho=a^2+b^2=2$ then $(\nabla_{\!\frac{\partial}{\partial t}}
\nabla_{\!\frac{\partial}{\partial t}} \Omega^1_2)_{\mid_{t=0}} =
-288\, (e^{12} - e^{34})$ and
$$
\begin{array} {rcl}
(\nabla_{\!\frac{\partial}{\partial t}}
 \Omega^1_5)_{\mid_{t=0}} \!\!&\!\!\!=\!\!\!&\!\! 12\, b\,(e^{13}+e^{24}) - 12\, a\, (e^{14}
 - e^{23}) -96\,(e^{15}+e^{46}),\\[8pt]
 (\nabla_{\!\frac{\partial}{\partial t}}
 \Omega^2_5)_{\mid_{t=0}} \!\!&\!\!\!=\!\!\!&\!\! 12\, a\, (e^{13}+e^{24}) + 12\, b\,
 (e^{14}-e^{23})-96\,(e^{25}+e^{36}),
 \end{array}
 $$
which implies that $\nabla_{\!\frac{\partial}{\partial t}}
\nabla_{\!\frac{\partial}{\partial t}} \Omega^1_2, \Omega^1_3,
\Omega^1_4, \nabla_{\!\frac{\partial}{\partial t}} \Omega^1_5,
\Omega^1_6, \Omega^2_3, \nabla_{\!\frac{\partial}{\partial t}}
\Omega^2_5$ and $\Omega^2_6$ are linearly independent at $t=0$. For
the remaining case $\rho=a^2+b^2=6$, since
$$
(\nabla_{\!\frac{\partial}{\partial t}} \Omega^1_4)_{\mid_{t=0}} =
536\, e^{14} -96\, a\, (e^{15}+e^{46}) + 40\, e^{23} + 96\, b\,
(e^{25}+ e^{36})  - 576\, e^{56},
$$
we have that the forms $\Omega^1_2, \Omega^1_3,
\nabla_{\!\frac{\partial}{\partial t}} \Omega^1_4, \Omega^1_5,
\Omega^1_6, \Omega^2_3, \Omega^2_5$ and $\Omega^2_6$ are independent
at $t=0$. Therefore, any left-invariant hypo-contact structure on
the Lie group $H_4$ gives rise to a metric with holonomy $SU(3)$.

For the family \eqref{Azero-ecus-5} a solution of the hypo evolution
equations is given by
$$
\begin{array} {lll}
\eta(t) \!\!&\!\!\!= \frac 12 f'(t) e^5, \quad\quad\quad &
\omega_1 (t) = f(t) \,  (e^{12} +   e^{34}),\\[6pt]
\omega_2(t) \!\!&\!\!\!= \frac 12  f^2(t) f'(t)\,  e^{13} - \frac
2{f'(t)} e^{24},\quad\quad\quad & \omega_3 (t) = f(t) \, (e^{14} +
e^{23}),
\end{array}
$$
where $f(t)$ satisfies the differential equation
$$
2 r^2 +{f'}^3+ f f' f'' =0
$$
with initial conditions $f(0) = 1$ and $f'(0)=2$. After performing a
first integration, one obtains the first order differential equation
$$
f'(t)=\frac{\left(8+2r^2-2r^2f^3(t)\right)^{1/3}}{f(t)}
$$
with initial conditions $f(0) = 1$, so there exists a unique
solution $f(t)$ defined on some open neighbourhood $I$ around $t=0$.
A similar computation as above shows that the holonomy group of the
metric associated to the corresponding integrable $SU(3)$-structure
on $H_5\times I$ is again equal to $SU(3)$. In fact$$
\begin{array} {lll}
(\Omega^1_2)_{\mid_{t=0}} \!\!&\!\!\!= \frac{(r^2+4)(r^2-4)}{16}
(e^{12} - e^{34}),
\quad\quad\quad  &(\Omega^1_3)_{\mid_{t=0}} = - \frac{(r^2+4)^2}{16} (e^{13} + e^{24}) ,\\[7pt]
(\Omega^1_4)_{\mid_{t=0}} \!\!&\!\!\!= \frac{r^2+4}{2} (
\frac{r^2-8}{4} e^{14} - e^{23} - \frac{r^2-12}{4} e^{56}),
\quad\quad\quad& (\Omega^1_5)_{\mid_{t=0}} = - \frac{(r^2+4)(5r^2-12)}{16} (e^{15} + e^{46})  \\[7pt]
(\Omega^1_6)_{\mid_{t=0}} \!\!&\!\!\!= \frac{3(r^2+4)^2}{16} (e^{16}
- e^{45}), \quad\quad\quad& (\Omega^2_3)_{\mid_{t=0}} = -
\frac{r^2+4}{2} ( e^{14} - \frac{r^2-8}{4}
e^{23} + \frac{r^2-12}{4} e^{56}),\\[7pt]
(\Omega^2_5)_{\mid_{t=0}} \!\!&\!\!\!= \frac{3(r^2+4)^2}{16} (e^{25}
+ e^{36}), \quad\quad\quad& (\Omega^2_6)_{\mid_{t=0}} = -
\frac{(r^2+4)(5r^2-12)}{16} (e^{26} - e^{35}),
\end{array}
$$
where again $e^6$ denotes the 1-form $dt$ evaluated at $t=0$.
Moreover,
$$
\begin{array} {rcl}
(\nabla_{\!\frac{\partial}{\partial t}} \Omega^1_2)_{\mid_{t=0}}
\!\!&\!\!\!=\!\!\!&\!\! \frac{(r^2+4)(r^4-2r^2+16)}{8}(e^{12} -
e^{34}),\\[7pt]
(\nabla_{\!\frac{\partial}{\partial t}} \Omega^1_5)_{\mid_{t=0}}
\!\!&\!\!\!=\!\!\!&\!\! -\frac{(r^2+4)(5r^4-8r^2+48)}{8}(e^{15} +
e^{46}),\\[7pt]
(\nabla_{\!\frac{\partial}{\partial t}} \Omega^2_3)_{\mid_{t=0}}
\!\!&\!\!\!=\!\!\!&\!\! -\frac{r^2+4}{4} \left( (r+4)(r-4) e^{14}
- (r^4-3r^2+32)e^{23} + (r^4-4r^2+48) e^{56} \right),\\[7pt]
(\nabla_{\!\frac{\partial}{\partial t}} \Omega^2_6)_{\mid_{t=0}}
\!\!&\!\!\!=\!\!\!&\!\! -\frac{(r^2+4)(5r^4-8r^2+48)}{8}(e^{26} -
e^{35}).
\end{array}
$$
A direct calculation shows that, for each $r$, eight of the twelve
2-forms above are linearly independent.

Finally, for the the family \eqref{Azero-ecus-1}, a solution of the
hypo evolution equations is given by
$$
\begin{array} {lll}
\eta(t) \!\!&\!\!\!= \frac 12 f'(t) e^5, \quad\quad\quad &
\omega_1 (t) = \frac 12 f^2(t) f'(t) \,  e^{12} + \frac 2{f'(t)} \,  e^{34},\\[6pt]
\omega_2(t) \!\!&\!\!\!=  \frac12 f^2(t) f'(t)\, e^{13} - \frac
2{f'(t)}\, e^{24}, \quad\quad\quad & \omega_3 (t) =  f(t) \, (e^{14}
+ e^{23}),
\end{array}
$$
where $f(t)$ satisfies the differential equation
$$
12 r^2 +f {f'}^4 +f^2 {f'}^2 f'' =0
$$
with initial conditions $f(0) = 1$ and $f'(0)=2$. Equivalently,
$f(t)$ must satisfy the first order differential equation
$$
f'(t)=\frac{2}{f(t)} \sqrt{1+r^2-r^2f^3(t)}
$$
with initial conditions $f(0) = 1$, so there exists a unique
solution $f(t)$ defined on some open neighbourhood $I$ around $t=0$.
One can prove that the holonomy of the resulting metric on
$H_2\times I$ is again $SU(3)$. \QED\medskip

\section{Metrics with holonomy $G_2$}  \label{half-flat}

Let $H$ be a simply connected solvable Lie group of dimension $5$
with a left-invariant hypo-contact structure. In order to prove
Theorem \ref{G2-metrics}, we study first the induced  half-flat
structures on the total space of a circle bundle over $H$. In
particular, we will show that many hypo-contact structures on $H$
define not only the natural half-flat structure on  the trivial
bundle $H_5 \times \R$ but also another half-flat structure on a
non-trivial  $S^1$-bundle, which allows us to construct  a metric
with holonomy $G_2$.

Let us recall that Hitchin in \cite{H} proved that if $M$ is a
$6$-manifold with a half-flat structure $(F,\Psi_+,\Psi_-)$ which
belongs to a family $(F(t),\Psi_+(t),\Psi_-(t))$ of half-flat
structures on $M$, for some real parameter $t$ lying in some
interval $I=(t_-, t_+)$, satisfying the evolution equations
\begin{equation} \label{Hitchineveq}
\left\{
\begin{array} {l}
\partial_t \Psi_+(t) = \hat d F(t),\\[3pt]
F(t) \wedge \partial_t (F(t)) = - \hat d \Psi_-(t),
\end{array} \right.
\end{equation}
then $M\times I$ has a Riemannian metric whose holonomy is contained
in $G_2$. In fact, it is easy to check that the $4$-forms $\varphi$
and $* \varphi$ given by
$$
\varphi = F(t) \wedge dt + \Psi_+(t), \quad * \varphi = \psi_-(t)
\wedge dt+ \frac{1}{2} F(t)^2,
$$
are closed.

Next, we show that a solution of (hypo) evolution equations produces
a solution of Hitchin evolution equations. Let $N$ be a $5$-manifold
with a hypo structure $(\eta, \omega_i)$ which belongs to a
one-parameter family of hypo structures $(\eta(t),\omega_i(t))$, for
some real parameter $t \in I$, satisfying the (hypo) evolution
equations \eqref{vol-1}. Then, we know that an integrable
$SU(3)$-structure $(F,\Psi_+,\Psi_-)$ on $M=N \times I$ is given by
$$
F = \eta(t) \wedge dt + \omega_3(t),\quad \Psi = (\omega_1(t) + i
\omega_2(t)) \wedge (\eta(t) + i dt).
$$
\noindent On the other hand, Proposition
\ref{from-hypo-to-half-flat} implies that the $SU(3)$-structure
$(F,\Psi_+,$ $\Psi_-)$ on $M=N \times \R$ given by
$$
F = \lambda\, \omega_1 + \mu\, \omega_2 + \eta \wedge e^{6},\quad
\Psi_+ = ( - \mu\, \omega_1  + \lambda\, \omega_2)  \wedge \eta -
\omega_3 \wedge e^6,\quad \Psi_- = (- \mu\, \omega_1 + \lambda\,
\omega_2)  \wedge e^6  + \omega_3 \wedge \eta,
$$
is half-flat for $\lambda$, $\mu \in \R$ with $\lambda^2 +\mu^2=1$.
Moreover, using again Proposition \ref{from-hypo-to-half-flat}, we
have the one-parameter family of half-flat structures
$(F(t),\Psi_+(t),\Psi_-(t))$ on $M=N \times \R$ defined by
\begin{equation} \label{hypo-Hitchin}
\left\{
\begin{array} {l}
F(t) = \lambda\, \omega_1(t) + \mu\, \omega_2(t) + \eta(t) \wedge
e^{6},\\[3pt]
\Psi_+(t) = ( - \mu\, \omega_1(t)  + \lambda\, \omega_2(t))  \wedge
\eta(t) - \omega_3(t) \wedge
e^6,\\[3pt]
\Psi_- (t) = (- \mu\, \omega_1(t) + \lambda\, \omega_2(t))  \wedge
e^6 + \omega_3(t) \wedge \eta(t),
\end {array} \right.
\end{equation}
where $e^6 (t) = e^6$, for any $t$.

\begin{proposition} The family $(F(t),\Psi_+(t),\Psi_-(t))$ of half-flat structures
on $M=N \times \R$ given by \eqref{hypo-Hitchin} is a solution of
the Hitchin evolution equations \eqref{Hitchineveq}.
\end{proposition}

\begin{proof}
Clearly, $\hat d F(t)= \lambda\, \hat d\omega_1 (t) + \mu\,  \hat
d\omega_2 (t)$ and from equations \eqref{vol-1} we have
$\partial_t{\Psi_+} (t)  = \hat d F (t)$.
Moreover, since $\hat d \Psi_-(t) = (- \mu\, \hat d \omega_1(t) +
\lambda\, \hat d \omega_2 (t)) \wedge e^6 +  \hat d (\omega_3 (t)
\wedge \eta(t))$ and
$$F(t) \wedge \partial_t  F(t) = \frac 12 \partial_t (( \lambda\, \omega_1 (t) + \mu\, \omega_2(t)) ^2)
+ \left[ \lambda\,  \partial_t  (\omega_1 (t) \wedge \eta(t))  +
\mu\, \partial_t  (\omega_2 (t) \wedge \eta(t)) \right] \wedge
e^6,$$ the second equation in \eqref{Hitchineveq} is satisfied if
and only if
$$
\hat d (\omega_3 (t)\wedge \eta(t)) =  - \frac 12 \partial_t ((
\lambda\, \omega_1 (t) + \mu\, \omega_2(t)) ^2).$$ But, from
\eqref{wedge-i-j} and $\lambda^2+\mu^2=1$ we get $(\lambda\,
\omega_1(t)+\mu\, \omega_2(t))^2 =\omega_3(t)\wedge\omega_3(t)$, and
therefore
$$
\frac 12\partial_t((\lambda\, \omega_1(t)+\mu\, \omega_2(t))^2)=
\frac 12\partial_t(\omega_3(t) \wedge\omega_3(t))= \omega_3(t)
\wedge\partial_t\omega_3(t)=-
\omega_3(t)\wedge\hat{d}\eta(t)=-\hat{d}(\omega_3(t)\wedge\eta (t)).
$$
\end{proof}

We must notice that this result, which is also used in
\cite{ConTom}, implies that the holonomy of the resulting
$G_2$-metric on $M \times I$ is contained in $SU(3)$, because it is
actually a product metric. This fact justifies our study of
half-flat structures on  non-trivial circle bundles  (see
Remark~\ref{clasif-nontriv} below).

Let ${\mathfrak h}$ be a solvable $5$-dimensional Lie algebra with a
hypo structure $(\eta, \omega_1, \omega_2,\omega_3)$. Consider the
extension ${\mathfrak k} = {\mathfrak h} \oplus \R e_6$, with $e_6$
such that the Jacobi identity is satisfied. The $SU(3)$-structure on
$\mathfrak k$ defined by
$$
F = \lambda\, \omega_1 + \mu\, \omega_2 + e^{56},\quad \Psi_+ = ( -
\mu\, \omega_1  + \lambda\, \omega_2)  \wedge e^5 - \omega_3 \wedge
e^6,\quad \Psi_- = (- \mu\, \omega_1 + \lambda\, \omega_2)  \wedge
e^6 + \omega_3 \wedge e^5,
$$
with $\lambda^2 + \mu^2 =1$, is half-flat if and only if $d(F \wedge
F) = 2 (\lambda\, \omega_1 + \mu\, \omega_2)  \wedge e^5 \wedge (d
e^6)= 0$ and $d(\Psi_+) =  - \omega_3 \wedge (d e^6)=0$. From these
equations one has that  $$ d e^6 = a_1 e^{12} + a_2 e^{13} + a_3
(e^{14} - e^{23}) + a_5 e^{24}  + a_6 e^{34},
$$
with $\lambda (a_1 +  a_6) + \mu (a_2 - a_5) =0$. Then $d(d e^6) =0$
only in the following cases:
\begin{enumerate}
\item $d e^6 = 0$ for all the families;
\item $d e^6 = a_1 e^{12} + a_2 e^{13}$, with $\lambda a_1 + \mu a_2 =0$ for the family
\eqref{Azero-ecus-1};
\item $d e^6 = a_2 ( - \frac{a}{r} e^{12} + e^{13})$ for the family
\eqref{Azero-ecus-3} with $\mu =\frac{a}{r} \lambda$;
\item $d e^6 = a_1 e^{12}$ for the family \eqref{Azero-ecus-4} with $\lambda
=0$;
\item $d e^6 = a_2 e^{13}$ for the family \eqref{Azero-ecus-5} with
$\mu=0$;
\item $d e^6 = a_1 (e^{12} + \frac{a}{r} e^{13})$ for the family \eqref{Azero-ecus-7} with $\lambda =
- \frac{a}{r} \mu$.
\end{enumerate}

\begin{remark}\label{clasif-nontriv}
Notice that the previous cases $2$--$6$ give a classification of the
half-flat structures on ${\mathfrak k}$ which are a non-trivial
extension of the hypo structure on ${\mathfrak h}$. \end{remark}

\medskip

\noindent{\bf Proof of Theorem~\ref{G2-metrics}~:} For the
non-trivial $S^1$-bundle $K$ associated to  the family
\eqref{Azero-ecus-4}  with $\lambda =0, \mu = 1$ and $d e^6 = a_1
e^{12}$, one has that a solution of the evolution equations
\eqref{Hitchineveq}  is given by
$$
\begin{array}{l}
F(t) = f(t)(e^{13} - e^{24}) + k(t) h(t) e^{56},\\[6pt]
\Psi_+ (t) = -f(t)^2 k(t)^{2} e^{125} - e^{345} - f(t) h(t) (e^{146} + e^{236}),\\[5pt]
\Psi_-(t) =\displaystyle -f(t)^2 h(t) k(t) e^{126} -
\frac{h(t)}{k(t)} e^{346} + k(t) f(t) (e^{145} + e^{235}),
\end{array}
$$
where $f(t), k(t), h(t)$ are functions satisfying the system of
ordinary differential equations
$$
\left \{
\begin{array} {l}
(f h)'  =  2 k h,\\[5pt]
(f^2 k^2)'  = a_1 k h - 2(a^2+b^2) f,\\[5pt]
f f' =  2 k f + \frac{a_1 h}{2k},
\end{array}
\right.
$$
and the initial conditions $f(0) = k(0) = h(0) =1$. This system is
easily seen to be equivalent to
\begin{equation}\label{systemfamili25-hit}
f'=  \displaystyle  2k+\frac{a_1 h}{2kf},\quad\  h'  =
\displaystyle-\frac{a_1h^2}{2kf^2},\quad\
k'=\displaystyle-\frac{a^2+b^2+2k^3}{kf},
\end{equation}
and thus by the theorem on existence of solutions for a system of
ordinary differential equations, there exists an open interval $I$
containing $t =0$ on which the previous system admits a unique
solution $(f(t), k(t),h(t))$ satisfying the initial condition $f(0)
= k(0) = h(0) = 1$.

For $a = b=0$, the $5$-dimensional hypo-contact Lie algebra is the
nilpotent Lie algebra ${\mathfrak h}_1$ and a solution in this case
is given by
$$
a_1 = 2, \quad f(t) = (1+5t)^{\frac{3}{5}}, \quad h(t)
=(1+5t)^{-\frac{1}{5}}, \quad k(t) = (1 + 5t)^{-\frac {2}{5}}.
$$
The corresponding metric with holonomy $G_2$ that we obtain is the
one found in \cite{CFino}.

For $a^2 + b^2 \neq 0$ and $a_1 =2$, the corresponding metric $g(t)$
on $K \times I$, where $K$ has  structure equations
$$
\left\{
\begin{array}{rl}
d e^1=\!\!&\!\! d e^2= 0,\\[4pt]
d e^3=\!\!&\!\! a e^{13} + b e^{14} - b e^{23} + a e^{24} - (a^2+b^2) e^{25},\\[4pt]
d e^4=\!\!&\!\! b e^{13} - a e^{14} - (a^2+b^2) e^{15} + a e^{23} + b e^{24},\\[4pt]
d e^5 =\!\!&\!\! -2 e^{14} - 2 e^{23},\\[4pt]
d e^6 =\!\!&\!\! 2 e^{12},
\end{array}
\right.$$ is given by
$$
g(t) = f(t)^2 k(t)\bigl( (e^1)^2 + (e^2)^2\bigr) +
\frac{1}{k(t)}\bigl((e^3)^2 +  (e^4)^2\bigr) + k(t)^{2} (e^5)^2 +
h(t)^2 (e^6)^2 + (dt)^2.
$$
The metric $g(t)$ has holonomy $G_2$ for $(a, b)$  in  a small
neighbourghood around $(0,0)$, since  the solution $(f(t), k(t),
h(t))$ of the system \eqref{systemfamili25-hit} depends continuously
on the parameters $a$ and $b$, and for $a=b=0$ the holonomy of the
corresponding metric is $G_2$.

\smallskip

For the non-trivial extension on the Lie group $\widetilde K$
associated to  the family \eqref{Azero-ecus-5}   with $\mu =0,
\lambda = 1$ and  $d e^6 = a_2 e^{13}$, one has that a solution of
the evolution equations \eqref{Hitchineveq}  is given by
$$
\begin{array}{l}
F(t) = f(t)(e^{12} + e^{34}) + k(t) h(t) e^{56},\\[6pt]
\Psi_+ (t) = f(t)^2 k(t)^{2} e^{135} - e^{245} - f(t)  h(t) (e^{146} + e^{236}),\\[5pt]
\Psi_-(t) = \displaystyle f(t)^2 h(t) k(t)  e^{136} -
\frac{h(t)}{k(t)} e^{246} + f(t) k(t)(e^{145} + e^{235}),
\end{array}
$$
where $f(t), k(t), h(t)$ are functions satisfying  the system of
ordinary differential equations
\begin{equation} \label{systemfamili26-hit}
f'  = \displaystyle 2k-\frac{a_2 h}{2kf},\quad\  h'  =
\displaystyle\frac{a_2h^2}{2kf^2},\quad\
k'=-\displaystyle\frac{r^2+4k^3}{2kf},
\end{equation}
and the initial conditions $f(0) = k(0) = h(0) =1$. Thus by the
theorem on existence of solutions for a system of ordinary
differential equations, there is an open interval $I$ containing $t
=0$ on which the previous system has a unique solution $(f(t),
k(t),h(t))$ satisfying the initial condition $f(0)= k(0) = h(0) =
1$. Since the system \eqref{systemfamili26-hit} for $r=0$  and
$a_2=-2$ coincides with the system \eqref{systemfamili25-hit} for
$a=b=0$ and $a_1=2$, we can  use the same argument as for the
previous family to prove that in a small neighbourghood around $0$
the corresponding metric $\tilde g(t)$ on $\widetilde K \times I$
has holonomy $G_2$. In this case, $\widetilde K$ has structure
equations
$$
\left\{
\begin{array}{llll}
\!\!\!&\!\!\! d e^1= 0, \ \ & d e^2=  r e^{34} + {r^2 \over 2}  e^{35}, \ \quad & d e^3=  r e^{13}, \\[4pt]
\!\!\!&\!\!\! d e^4= - {r^2\over 2} e^{15} + r e^{23}, \ \ & d e^5 =
-2 e^{14} - 2 e^{23}, \ \quad & d e^6 = -2 e^{13},
\end{array}
\right.
$$
and the metric $\tilde g(t)$ is given, in terms of the basis $(e^1,
\ldots, e^6, dt)$  by
$$
\tilde g(t) = f(t)^2 k(t)\bigl( (e^1)^2 + (e^3)^2\bigr) +
\frac{1}{k(t)}\bigl((e^2)^2 + (e^4)^2\bigr) + k(t)^{2} (e^5)^2 +
h(t)^2 (e^6)^2 + (dt)^2.
$$
\QED

\begin{remark}
Note that the $6$-dimensional solvable Lie groups $K$ (with $a^2 +
b^2 \neq 0$) and $\widetilde K$ (with $r \neq 0$) are not
isomorphic, since ${\mathfrak k}^2=0$ for the first family while
$\tilde {\mathfrak k}^2 \neq 0$ for  the second one. For $a=b=0$ and
$r =0$ one gets the  same $6$-dimensional  nilpotent Lie group.
Moreover, taking into account the explicit isomorphisms given in the
proof of Theorem~\ref{clasification}, one can see that for any
$(a,b)\not= (0,0)$ the solvable Lie algebra ${\mathfrak k}$ is
isomorphic to
$$
d\alpha^1=-\alpha^{14},\quad d\alpha^2=-\alpha^{25},\quad
d\alpha^3=\alpha^{34}+\alpha^{35},\quad d\alpha^4=d\alpha^5=0,\quad
d\alpha^6=\alpha^{45},
$$
and that for any $r\not=0$ the Lie algebra $\widetilde {\mathfrak
k}$ is isomorphic to the product ${\mathfrak h}_5\times \mathbb{R}$,
${\mathfrak h}_5$ being the solvable Lie algebra of
Theorem~\ref{clasification}.
\end{remark}

\begin{remark}
From the proof of Theorem 1.3 above, we see that one can ensure that
the holonomy of our examples equals $G_2$ when the parameters
$a,b,r$ are sufficiently close to $0$. To our knowledge, there is no
similar result in the literature about existence of metrics of
holonomy equal to $G_2$ neither on $K\times I$ nor on $\widetilde
K\times I$, so in this sense our result provides new spaces of $G_2$
holonomy.
\end{remark}

\medskip
\noindent {\bf Acknowledgments.}  We would like to thank Stefan
Ivanov and Simon Salamon
 for useful discussions on the hypo-contact structures.
This work has been partially supported through grants MCyT (Spain)
MTM2005-08757-C04-02, Project UPV 00127.310-E-15909/2004,  Project
MIUR (Italy) \lq \lq Riemannian Metrics  and Differentiable
Manifolds" and by GNSAGA.

{\small

\vspace{0.15cm}

\noindent{\sf L.C. de Andr\'es and M. Fern\'andez:} Departamento de
Matem\'aticas, Facultad de Ciencia  Tecnolog\'{\i}a, Universidad del Pa\'{\i}s
Vasco, Apartado 644, 48080 Bilbao, Spain.

{\sl E-mail L.C. de Andr\'es:} luisc.deandres@ehu.es

{\sl E-mail M. Fern\'andez:} marisa.fernandez@ehu.es

\vspace{0.15cm}

\noindent{\sf A. Fino:} Dipartimento di Matematica, Universit\`a di
Torino, Via Carlo Alberto 10, Torino, Italy.

{\sl E-mail:} annamaria.fino@unito.it

\vspace{0.15cm}

\noindent{\sf L. Ugarte:} Departamento de Matem\'aticas \!-\!
I.U.M.A., Universidad de Zaragoza, Campus Plaza San Francisco, 50009
Zaragoza, Spain.

{\sl E-mail:} ugarte@unizar.es}


\begin{thebibliography}{33}

\bibitem{AS} V. Apostolov, S. Salamon, Kahler reduction of metrics with holonomy
$G_2$, \emph{Commun. Math. Phys.}  {\bf 246} (2004), 43--61.

\bibitem{BV} L.  Bedulli, L. Vezzoni, Torsion of $SU(2)$-structures and Ricci
curvature in dimension $5$, preprint DG/0702790v3.

\bibitem{Blair} D.E. Blair, Riemannian geometry of contact and symplectic manifolds,
\emph{Progress in Mathematics 203}, Birkh\"{a}user Boston, Inc., Boston,
MA, 2002.

\bibitem{BGal}
C.Boyer, K. Galicki, {\em 3-Sasakian manifolds}, Surveys Diff. Geom.
{\bf 6} (1999), 123--184.

\bibitem{BGM} C.P. Boyer, K. Galicki, P. Matzeu, On eta-Einstein Sasakian geometry,
\emph{Commun. Math. Phys.} {\bf 262} (2006), 177--208.

\bibitem{CFino} S. Chiossi, A. Fino, Conformally parallel
$G_2$ structures on a class of solvmanifolds, \emph{Math. Z.} (4)
{\bf 253} (2006), 825--848.

\bibitem{CS} S. Chiossi, S. Salamon, The intrinsic torsion of $SU(3)$ and $G_2$ structures,
In: Gil-Medrano, O., Miquel, V. (eds.), Differential geometry,
Valencia, 2001,  World Sci. Publishing, River Edge, NJ, 2002,
115--133.

\bibitem{Con} D. Conti, Special holonomy and hypersurfaces, Ph.D. thesis, Scuola
Normale Superiore, Pisa, 2005.

\bibitem{ConS} D. Conti, S. Salamon, Generalized Killing spinors in dimension 5,
\emph{Trans. Amer. Math. Soc.} {\bf 359} (2007), 5319--5343.

\bibitem{ConTom} D. Conti, A. Tomassini, Special symplectic six-manifolds,
\emph{Q. J. Math.} {\bf 58} (2007), 297--311.

\bibitem{D} A. Diatta, Left invariant contact structures on Lie groups,
to appear in \emph{Differential Geom. Appl.}, preprint DG/0403555.

\bibitem{GLPS} G. W. Gibbons,  H.  L\" u, H., C.N. Pope,  K.S.  Stelle,
Supersymmetric domain walls from metrics of special holonomy,
\emph{Nuclear Phys. B}  {\bf 623}  (2002), 3--46.

\bibitem{H} N. J. Hitchin, Stable forms and special metrics. In: Fern\'andez, M., Wolf J. (eds.),
Global differential geometry: the mathematical legacy of Alfred Gray
(Bilbao, 2000), Contemp. Math. {\bf 288}, Amer. Math. Soc.,
providence, RI, 2001, 70--89.

\bibitem{Kobayashi} S. Kobayashi, Principal fibre bundles
with the $1$-dimensional toroidal group, \emph{T\^ohoku Math. J.}
(56) {\bf 8} (1956), 29--45.

\bibitem{Ok} M. Okumura, Some remarks on spaces with a certain contact structure,
\emph{T\^ohoku Math. J.} {\bf 14} (1962), 135--145.

\bibitem{TV} A. Tomassini, L. Vezzoni, Contact Calabi-Yau manifolds
and special Legendrian submanifolds, \emph{Osaka J. Math.} {\bf 45}
(2008), 127--147.

\bibitem{U} L. Ugarte, Hermitian structures on six dimensional
nilmanifolds, \emph{Transform. Groups} {\bf 12} (2007), 175--202.

\end{thebibliography}
\end{document}